\documentclass[authoryear,12pt]{article}
\usepackage{a4wide,epic}
\usepackage[T1]{fontenc}
\usepackage{graphicx}
\usepackage{amsmath}
\usepackage{amssymb}
\usepackage{amsfonts}
\usepackage{epsfig}
\usepackage{enumitem}
\usepackage{color}
\usepackage{picture}
\usepackage[colorlinks=true]{hyperref}

\newtheorem{theorem}{Theorem}
\newtheorem{proposition}{Proposition}
\newtheorem{remark}{Remark}
\newtheorem{lemma}{Lemma}
\newtheorem{example}{Example}

\newtheorem{corollary}{Corollary}
\newtheorem{definition}{Definition}
\newtheorem{assumption}{Assumption}
\newenvironment{proof}{{{\bf Proof:}\mbox{ }}}{{\mbox{ }\hfill $\Box$}}
\newenvironment{algo}{{\bf Algorithm}\ttfamily\selectfont}{}
\newenvironment{keyword}{{\noindent{\bf keywords:}\mbox{ }}}{}

\newcommand{\e}{\mathbf{e}}
\newcommand{\V}{\mathcal{V}}
\newcommand{\E}{\mathcal{E}}
\newcommand{\G}{\mathcal{G}}
\newcommand{\M}{\mathcal{M}}
\newcommand{\R}{\ensuremath{\mathbb{R}}}
\newcommand{\N}{\mathbb{N}}
\newcommand{\so}{\ensuremath{\mathfrak{so}}}
\newcommand{\eps}{\varepsilon}

\newcommand{\vl}{vl}
\newcommand{\vr}{vr}
\newcommand{\ed}{\text{ed}}
\newcommand{\tightlist}{\setlength{\itemsep}{0cm} \setlength{\parskip}{0cm}}
\DeclareMathOperator{\tr}{tr}
\DeclareMathOperator{\sk}{skew}

\DeclareMathOperator{\grad}{grad}
\DeclareMathOperator{\st}{stab}
\DeclareMathOperator{\rk}{rank}
\DeclareMathOperator{\diag}{diag}

\DeclareMathOperator{\vect}{vec}
\DeclareMathOperator{\dist}{dist}

\newcommand{\refy}{reference~}
\newcommand{\news}{\sigma}

\begin{document}

\definecolor{gray}{rgb}{0.6,0.6,0.9}


\title{Synchronization with partial state coupling on $SO(n)$\footnote{This paper presents research results of the Belgian Network DYSCO (Dynamical Systems, Control, and Optimization), funded by the Interuniversity Attraction Poles Program, initiated by the Belgian State, Science Policy Office. The scientific responsibility rests with its authors.
\newline
A preliminary version of this work has been presented at the 48th IEEE Conference on Decision and Control, Shanghai, China, 2009~\cite{myCDC}.
\newline
The authors want to thank Prof.~Rodolphe Sepulchre for stimulating discussions.
}}

\author{A. Sarlette\footnote{Corresponding author. Systems and Modeling, Department of Electrical Engineering and Computer Science, University of Li\`ege, 4000 Li\`ege, Belgium. FNRS postdoctoral fellow; visiting at Centre Automatique et Syst\`emes, Mines ParisTech, Paris, France. e-mail: alain.sarlette@ulg.ac.be}~~and C. Lageman\footnote{Institut f\"ur Mathematik, Universit\"at W\"urzburg, 97070 W\"urzburg, Germany. e-mail: christian.lageman@mathematik.uni-wuerzburg.de}}

\maketitle

\begin{abstract}
This paper studies autonomous synchronization of $k$ agents whose states evolve on $SO(n)$, 
but which are only coupled through the action of their states on one ``\refy vector'' in $\R^n$ for each link. Thus each link conveys only partial state information at each time, and to reach synchronization agents must combine this information over time or throughout the network.
A natural gradient coupling law for synchronization is proposed. Extensive convergence analysis of the coupled agents is provided, both for fixed and time-varying \refy vectors.
The case of $SO(3)$ with fixed \refy vectors is discussed in more detail.
For comparison, we also treat the equivalent setting in $\R^n$, i.e.~with states in $\R^n$ and connected agents comparing scalar product of their states with a \refy vector.
\end{abstract}




\section{Introduction}\label{sec:Intro}

Synchronization and collective phenomena have recently drawn considerable attention. Inspired by physical and natural systems
(see e.g.~\cite{Reynolds1,Sync,VICSEK}), the control community has developed an interest in self-coordination of multi-agent systems; applications include communication networks and vehicle formations
(see e.g.~\cite{BULLO,HendrickxPersistence,HOPFIELD,JADBABAIE,JandK,LEONARDgen,MCINNES,SujitThesis,TsitsiklisThesis}). 
The literature addresses aspects like optimal configuration, collision avoidance, nonlinear dynamics, interaction graph structure, distributed controller tuning, etc.

An important distinction must be made between \emph{decentralized} or \emph{distributed feedback control}, where controllers with (sometimes implicit) reference inputs are designed for a set of coupled subsystems \cite{DahlehSI,SiljakBook}, and \emph{autonomous coordination}, where systems evolve independently of any external command and a collective behavior emerges from their coupling only. Our work belongs to the latter framework. 

A fundamental autonomous collective behavior, called ``synchronization'' or ``consensus'', is that agents connected by a restricted set of pairwise interaction links reach agreement on a common state value.
This problem has been studied mainly on basis of a standard ``linear consensus algorithm'' for states in $\mathbb{R}^n$, see e.g.~\cite{JADBABAIE,olfati,TsitsiklisThesis}. A more general convexity argument is proposed by~\cite{MOREAU}. 
In~\cite{ss:08}, consensus is extended to nonlinear manifolds, motivated by problems involving e.g.~data on the sphere, agents representing subspaces, or rigid body attitudes. Synchronization of satellite attitudes, evolving on the space 
$SO(3)$ of rotation matrices, has attracted much attention,
both in presence of a reference or leader~\cite{NorwayLeaderFollower,SO3book,BeardOnSats,VanDyke1} and in fully cooperative (i.e. leader- and reference-less) framework~\cite{SujitThesis,WRen1,MY7}, mostly with second-order mechanical dynamics.
All these papers consider \emph{full state exchange}\footnote{Or at least full configuration exchange, for second-order systems.} between interacting agents. Consensus with partial state coupling in linear systems is studied in \cite{ScardoviLinearSyst,Emre08}: 
agents are coupled by linear outputs that are essentially equivalent to the projection of the state on \emph{the same subspace} for all pairs of interacting agents and for all time;
state agreement is reached thanks to detectable zero-input agent dynamics.

The present paper introduces two innovations in this regard.
First, it studies synchronization with \emph{partial state coupling on the Lie group $SO(n)$}: a basic coupling law for synchronization on $SO(n)$ (see~\cite{ss:08}) is extended to a setting where information coupling the agents only contains the action of their state on a ``\refy vector'' in $\R^n$. This output map is inspired by~\cite{piovan:08}; however~\cite{piovan:08}, to focus on noise reduction, only studies the case of $SO(2)$ --- that is the circle --- where the state is fully determined by a scalar output value.
Second, we consider a setting where agents' zero-input dynamics vanish, but partial state information involves an (a priori) 
\emph{different \refy vector for each interaction link}.
Full state synchronization is then recovered by partial state exchanges through the whole network.
For comparison, corresponding results for states evolving on $\R^n$ are also given. 
The main goals of the paper are
\begin{itemize}
	\item for practical purposes, to propose results allowing synchronization of e.g.~satellite attitudes with reduced information transfer;
	\item from a theoretical viewpoint, to illustrate implications of exchanging \emph{different} pieces of information along different links in coordinating networks, in particular for a nonlinear state space --- the Lie group $SO(n)$.
Unlike what first intuition might suggest, the fact that each agent shares information about its whole state with the network is not sufficient for synchronization.
\end{itemize}

Contributions are presented as follows. Section~\ref{sec:statement} formalizes the problem, proposes a natural cost function and derives a gradient law on $SO(n)$ that couples the agents via available information.
The analog problem on $\R^n$ is also introduced. 
Section~\ref{sec:2agents} starts the convergence analysis with 2 agents; this single-link case admits a complete analysis and mirrors the behavior of tracking or observer algorithms. With a varying \refy vector, states synchronize under persistent excitation: 
the ``averaged output map'' gives access to full state. A single fixed \refy vector is not sufficient to synchronize the states.
Section~\ref{sec:fixed} presents a detailed analysis for networks of $k>2$ agents.
Section~\ref{ssec:timvar} gives the result for time-varying \refy vectors, adding a technical assumption w.r.t.~the $2$-agent case. Local state synchronization for fixed \refy vectors is characterized in Section~\ref{ssec:fix}. We give (i) a necessary condition in terms of individual agent properties, (ii) a necessary and sufficient condition on $\R^n$ and a sufficient condition on $SO(n)$ in terms of the rank of a generalized Laplacian matrix including the effect of \refy vectors, and (iii) a tighter sufficient condition on $SO(n)$ which involves a matrix of size $k n^2$ for $k$ agents. Several (counter-)examples are given; in particular we show that, unlike with full state exchange, there are locally stable equilibria different from state synchronization even for an all-to-all interaction graph. 
Section~\ref{sec:SO3} studies $SO(3)$ with fixed \refy vectors in more detail, and proposes an algorithm to check if a network structure satisfies a sufficient condition for output synchronization and state synchronization to coincide under generic \refy vectors.
We also draw attention to poor robustness of the setting where \refy vectors correspond to relative positions of the agents in $\R^3$.
Simulations are provided for illustration.


\subsection{Notation} \label{sec:notation}
$I_n$ is the identity matrix in $\R^{n \times n}$. 
$X^T$ is the matrix transpose of $X\in\R^{n\times m}$ and, for $m=n$, $\sk(X)=\frac{1}{2}(X-X^T)$. $S^{n-1}$ is the unit sphere in $\R^n$. 
The Euclidean norm of column-vector $y \in \R^n$ is denoted  $\Vert y \Vert_2 := \sqrt{y^T y}$; 
The Frobenius norm of $M \in \R^{n \times n}$ with entries $m_{ij}$ is $ \Vert M \Vert_F := \sqrt{\tr(M^T M)} = \sqrt{\sum_{ij} \; m_{ij}^2}$. 
For $X \in \R^{n \times m}$ 
with column vectors $x_1, x_2,\ldots,x_m \in \R^n$, we write $\rk(X)=\rk(x_1,\, x_2,\,\ldots,\, x_m)$ its rank and $\vect(X) \in \R^{nm}$ the vector obtained by stacking $x_1, x_2,\ldots, x_m$ in a single column.
For a set $S\subset (\R^{n\times n})^k$ we denote $\vect(S) = \{\vect(X_1^T,\ldots,X_k^T)^T : (X_1,\ldots,X_k)\in S\} \subset \R^{kn^2}$.
Vector product of $x_1, x_2 \in \R^3$ is $x_1 \times x_2 \in \R^3$. Kronecker product of matrices is denoted $\otimes$.

$SO(n) \cong \{ Q \in \R^{n \times n} \, : \, Q^T Q = I_n , \, \det(Q)=1 \}$ is the Lie group of $n$-dimensional rotations; $\so(n) \cong \{ X \in \R^{n \times n} \, : \, X^T = -X \}$ is the Lie algebra of $SO(n)$, that is the tangent space to $SO(n)$ at identity.
We consider throughout the paper $SO(n)$ equipped with the standard biinvariant Riemannian metric from its embedding in $\R^{n\times n}$, 
i.e. $\langle Q\Omega, Q \Theta\rangle = \tr(\Omega^T\Theta)$, $\{Q\Omega,Q\Theta\}\subset T_Q SO(n)$, $\{\Omega,\Theta\}\subset\so(n)$.
$SO(n)$ acts on $\R^n$ by matrix-vector multiplication $(Q,y)\mapsto Qy$.
The stabilizer of $y \in \R^n$ with respect to this action is $\st(y) = \{Q \in SO(n) : Q y = y \}$. For $y\neq 0$ this stabilizer is a subgroup homomorphic to $SO(n-1)$; in particular, $Q \in \st{y}$ if and only if $Q^T \in \st{y}$.

An undirected (finite) graph $\G(\V,\E)$
consists of $\V$ a finite set of vertices  
--- in the following representing the agents ---
and $\E$ a set of unordered vertex pairs called edges --- representing the undirected interaction links among agents. 
We denote by $\#\E$ the cardinality of $\E$,
and following a customary abuse of notation we write $(i,j)$ to actually represent the unordered pair $\{ i,j\}$.
$\G(\V,\E)$ can be represented by its adjacency matrix $A=(a_{ij})$ where $a_{ij}= 1$ if $(i,j) \in \E$  
and $a_{ij}=0$ otherwise. In particular, $a_{ij}=a_{ji}$ for all $i,j$, i.e.~$A$ is symmetric.
We consider $\G$ without self-loops (nodes between a vertex and itself), so $a_{ii}=0$.
In the following, vertices (i.e.~agents) are represented by integers: $\V=\{ 1,\ldots,k\}$.
Vertices $i$, $j$ linked by edge $e=(i,j)\in \E$ are denoted $i=\vl(e)\in \V$ and $j=\vr(e) \in \V$, with $i<j$.
The set of edges attached to a vertex $i$ is denoted $\ed(i)$.
The degree $\deg(i)$ of vertex $i$ is the cardinality of $\ed(i)$.


\section{Problem statement and coupling law}\label{sec:statement}


\subsection{Partial state coupling on $SO(n)$}

Consider $k$ agents with states $Q_i \in SO(n)$ for $i=1,2,\ldots,k$,
$n \geq 2$, e.g. representing orientations of $3$-dimensional rigid bodies for $n=3$.
The agents evolve according to
\begin{equation}\label{eq:dyn}
\tfrac{d}{dt}Q_i = Q_i \, u_i  \quad , \quad i=1,2,\ldots,k \, ,
\end{equation}
where $u_i \in \so(n)$ is agent $i$'s control input. 
This left-invariant dynamics has a physical interpretation: control inputs are angular velocities expressed in respective body frames. 
The goal is to reach \emph{state synchronization} by just appropriately \emph{coupling} agents through $u_i$.
\begin{definition}
State synchronization set $C_s$ is the set of states such that all agents have the same state, i.e.
$C_s := \{(Q_1,Q_2,\ldots,Q_k) \in (SO(n))^k \, : \, Q_1=Q_2=\ldots=Q_k\}$. 
\end{definition}

The present paper considers the following restrictions on the $u_i$.
\begin{itemize}[topsep=0mm, parsep=0mm, itemsep=0mm]
\item[\textbf{(R1)}] \emph{Network coupling:} admissible pairwise interactions are represented by an undirected graph $\G(\V,\E)$ such that $u_i$ only uses information from agents $j$ for which $(i,j) \in \E$.
\item[\textbf{(R2)}] \emph{Partial state coupling:} interacting agents $i$ and $j$ exchange outputs
\begin{equation}\label{OutputMap}
Q_i^T y_{ij} \text{ and } Q_j^T y_{ij}
\end{equation}
with \refy vectors $y_{ij} = y_{ji} \in S^{n-1}$ given but a priori unknown.\vspace{2mm}
\end{itemize}
R1 is classical in consensus problems. Our main originality is R2; see the introduction for a discussion of related work. A physical interpretation of (\ref{OutputMap}) is, see Fig.\ref{fig:OutputMap}: $i$ and $j$ communicate to each other the coordinates in their respective body frames of direction $y_{ij}$. One could e.g.~imagine satellites comparing the position of some star on their cameras or how they see each other; in the latter case, $y_{ij}$ would be their relative position vector in inertial frame and $Q_i^T y_{ij}$ expresses it in the body frame of satellite $i$.
\begin{definition}
Output synchronization set $C_o$ is the set of states such that both outputs are equal on each link, i.e.~$C_o := \{(Q_1,Q_2,\ldots,Q_k) \in (SO(n))^k \, : \,  Q_i^T y_{ij} = Q_j^T y_{ij}$ $\forall (i,j) \in \E \}$.
\end{definition}
Clearly, $C_s \subseteq C_o$. The main question, investigated in this paper, is under which conditions output synchronization implies state synchronization. Both fixed and time-varying $y_{ij}$ are considered.
The challenge starts for $n > 2$: for $n=2$, as studied in~\cite{piovan:08}, $SO(2)$ is isomorphic to the circle and (\ref{OutputMap}) is equivalent to full state information. Synchronization on manifolds with full state coupling is analyzed e.g.~in \cite{MYthesis,ss:08}.

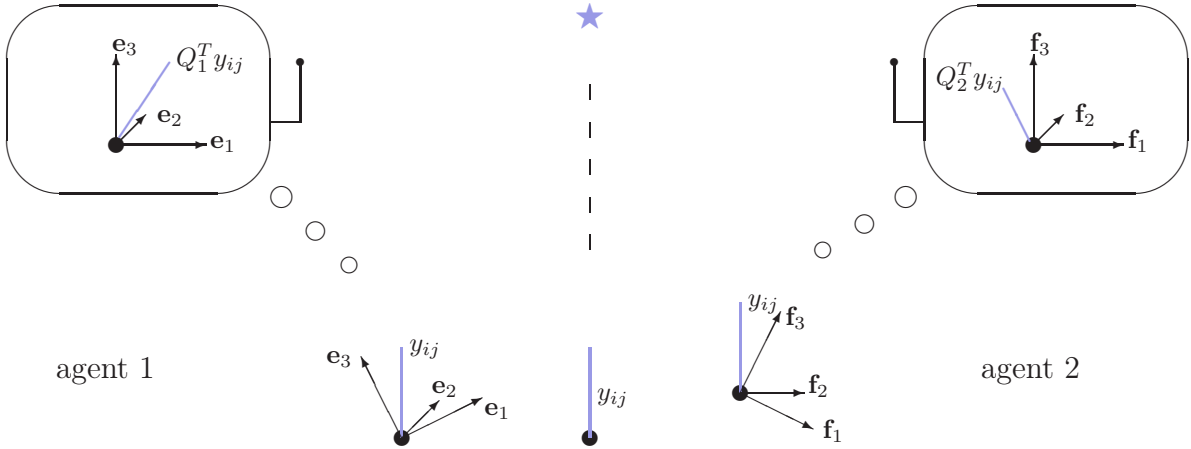
\begin{figure}[tbh]
	\setlength{\unitlength}{1mm}
	\begin{picture}(160,60)
	    \thinlines
		\put(80,0){\circle*{2}}
		\multiput(80,25)(0,5){5}{\line(0,1){2}}
		\thicklines
		\put(80,0){{\color{gray}\line(0,1){12}}}
		\thinlines
		\put(78,55){{\color{gray}$\bigstar$}}
		\put(81,5){{\footnotesize$y_{ij}$}}
		\multiput(20,45)(122,0){2}{\oval(35,25)}
		\multiput(41.5,42)(79,0){2}{\line(0,1){8}}
		\multiput(37.5,42)(83,0){2}{\line(1,0){4}}
		\multiput(41.5,50)(79,0){2}{\circle*{1}}
		\multiput(17,39)(122,0){2}{\vector(1,0){12}}
		\put(29.5,38.5){{\footnotesize$\mathbf{e}_1$}}
		\put(151.5,38.5){{\footnotesize$\mathbf{f}_1$}}
		\multiput(17,39)(122,0){2}{\vector(0,1){12}}
		\put(16.5,51.5){{\footnotesize$\mathbf{e}_3$}}
		\put(138.5,51.5){{\footnotesize$\mathbf{f}_3$}}
		\multiput(17,39)(122,0){2}{\circle*{2}}
		\multiput(17,39)(122,0){2}{\vector(1,1){4}}
		\put(22.5,42){{\footnotesize$\mathbf{e}_2$}}
		\put(144.5,42){{\footnotesize$\mathbf{f}_2$}}
		\put(9,8){\mbox{agent 1}}
		\put(55,0){\circle*{2}}
		\thicklines
		\put(55,0){{\color{gray}\line(0,1){12}}}
		\put(56,11.5){{\footnotesize$y_{ij}$}}
		\thinlines
		\put(55,0){\vector(2,1){10.7}}
		\put(55,0){\vector(-1,2){5.4}}
		\put(55,0){\vector(1,1){5}}
		\put(66,3){{\footnotesize$\mathbf{e}_1$}}
		\put(45,10){{\footnotesize$\mathbf{e}_3$}}
		\put(59,6){{\footnotesize$\mathbf{e}_2$}}
		\put(48,23){\circle{2}}
		\put(43.5,27.5){\circle{2.4}}
		\put(39,32){\circle{2.8}}
		\thicklines
		\put(17.5,40){{\color{gray}\line(2,3){6.6}}}
		\thinlines
		\put(25,49.5){{{\footnotesize$Q_1^T y_{ij}$}}}
		\put(132,8){\mbox{agent 2}}
		\put(100,6){\circle*{2}}
		\thicklines
		\put(100,6){{\color{gray}\line(0,1){12}}}
		\put(101,17.5){{\footnotesize$y_{ij}$}}
		\thinlines
		\put(100,6){\vector(2,-1){9.7}}
		\put(100,6){\vector(1,2){5.4}}
		\put(100,6){\vector(1,0){8.5}}
		\put(111,0){{\footnotesize$\mathbf{f}_1$}}
		\put(106,15){{\footnotesize$\mathbf{f}_3$}}
		\put(109,6){{\footnotesize$\mathbf{f}_2$}}
		\put(111,25){\circle{2}}
		\put(116.5,28.5){\circle{2.4}}
		\put(122,32){\circle{2.8}}
		\thicklines
		\put(138.5,39.5){{\color{gray}\line(-1,2){3.5}}}
		\thinlines
		\put(126,47){{{\footnotesize$Q_2^T y_{ij}$}}}				
	\end{picture}
	\caption{Interpretation of output map (\ref{OutputMap}) for $SO(3)$. Bottom center: two agents see a star in direction $y_{ij}$. Upper side panels: exchanged outputs $=$ expression of $y_{ij}$ in the respective body frames. Orthonormal frames ($\mathbf{e}_1,\mathbf{e}_2,\mathbf{e}_3$), ($\mathbf{f}_1,\mathbf{f}_2,\mathbf{f}_3$) indicate agent orientations.}\label{fig:OutputMap}
\end{figure}


\subsection{Gradient law on $SO(n)$} \label{ss:debut}

A gradient law for (at least local) synchronization on $SO(n)$ with \emph{full state coupling} is proposed in \cite{ss:08}. It expresses state disagreement by cost function
\begin{equation}\label{FullCost}
f_s = \frac{1}{2}\sum_{(i,j) \in \E} \, \Vert Q_i - Q_j \Vert_F^2 =  \sum_{(i,j) \in \E} \, (n-\tr(Q_i^TQ_j))
\end{equation}
and makes agents ``rotate towards each other'' according to
\begin{equation}\label{FullDyns}
\tfrac{d}{dt} Q_i = -\grad_{Q_i} (f_s) = Q_i \sum_{\{j : (i,j) \in \E\}} \, \sk(Q_i^T Q_j) \quad , \quad i=1,2,\ldots,k \, .
\end{equation}
The gradient is computed w.r.t.~the product metric on $SO(n)^k$ 
of the biinvariant Riemannian metric for $SO(n)$ embedded in $\mathbb{R}^{n \times n}$, see Section~\ref{sec:notation}. 
Algorithm (\ref{FullDyns}) requires agents to compare relative states, $Q_i^T Q_j$. 
For agents only comparing outputs (\ref{OutputMap}), we propose a gradient law 
which minimizes the output disagreement cost function
\begin{equation}\label{PartialCost}
f_o = \frac{1}{2}\sum_{(i,j) \in \E} \, \Vert Q_i^T y_{ij} - Q_j^T y_{ij} \Vert_2^2 = \sum_{(i,j) \in \E} \, (1-\tr(Q_i^T M_{ij} Q_j))
\end{equation}
where $M_{ij} := y_{ij} y_{ij}^T$ is a rank $1$ projector. This yields a
coupling law satisfying R1 and R2: 
\begin{equation}\label{alg:all}
\tfrac{d}{dt} Q_i = -\grad_{Q_i} (f_o) = \;\; Q_i \sum_{\{j : (i,j) \in \E\}} \, \sk(Q_i^T M_{ij} Q_j) \quad , \quad i=1,2,\ldots,k \, .
\end{equation}
Both dynamics (\ref{FullDyns}) and (\ref{alg:all}) are right-invariant w.r.t.~absolute orientation: if $(Q_1(t),\ldots,Q_k(t))$ is a solution then $(Q_1(t)R,\ldots,Q_k(t)R)$ is also a solution, for any constant $R \in SO(n)$.

\begin{remark}
We consider first-order dynamics for simplicity. Gradient laws are a basic tool that can be extended to higher-order systems, e.g.~applying force/torque that is the gradient of a potential in second-order mechanical systems; see e.g.~\cite{MYthesis} for extending (\ref{FullDyns}) to mechanical systems with Euler equation dynamics.
\end{remark}
\begin{remark}
Dynamics (\ref{alg:all}) can also be mapped to the space of outputs and viewed as a gradient in the submanifold of $M := (\, S^{n-1} \,)^{2\#\E}$ given by $N:=\{(r_{e_1},\ldots,r_{e_{\#\E}},s_{e_1},\ldots,s_{e_{\#\E}}) \colon  r_{e} = Q_{\vl(e)}^T y_e \text{ and } s_{e} = Q_{\vr(e)}^T y_e \text{ for some } (Q_1,Q_2,\ldots,Q_k) \in (SO(n))^k \, , \; \forall e\in \E \}$. 
The corresponding dynamics on $N$ extend naturally to gradient dynamics of 
$h_o = \sum_{e\in\mathcal E} \Vert r_e-s_e \Vert_2^2$ on $M$ for a suitable 
Riemannian metric. Explicitly, (\ref{alg:all}) on the output space translates to
$$
\tfrac{d}{dt}r_e  =\;  - \sum_{j \in \ed(\vr(e))}\sk(r_j \; s_j^T)\,r_e \quad \text{and} \quad
\tfrac{d}{dt}s_e  =\;  - \sum_{j \in \ed(\vl(e))}\sk(s_j \; r_j^T)\,s_e \quad \text{for } e \in \E \, .
$$
The appearance of $r_j \, s_j^T$, with $j\neq e$, in the gradient of $\|r_e-s_e\|_2^2$ expresses inter-dependence of the different outputs $r_j,s_j$ that correspond to the same state $Q_i$; this inter-dependence would be encoded in the geometry on $N$.
In this interpretation, output synchronization corresponds to the set
$N \, \bigcap \, \{(r_{e_1},\ldots,r_{e_{\#\E}},r_{e_1},\ldots,r_{e_{\#\E}}) : r_{e_j}\in S^{n-1}\}$.
Examining this intersection seems to be a challenging geometric problem. 
\end{remark}


\subsection{Gradient law with partial state coupling on $\R^n$}

An analog setting with agents on $\R^n$ will be simpler to investigate and serve for comparison in further sections. Consider $k$ agents with states $x_i \in \R^n$ for $i=1,2,\ldots,k$, e.g.~representing positions of agents in $\R^n$, $n \geq 1$. 
They evolve according to
$\tfrac{d}{dt}x_i = u_i$,
where $u_i \in \R^n$ must couple the agents to reach state synchronization $x_i = x_j$ $\forall i,j$. 
The coupling is subject to restrictions \textbf{(R1)} and \textbf{(R2)}: the network constraint remains unchanged; partial state coupling is imposed by comparing, for each edge $(i,j)$, the projection of $x_i$ and $x_j$ on some vector $y_{ij} \in S^{n-1} \subset \R^n$, i.e.~communicated outputs are $x_i^T y_{ij}$ and $x_j^T y_{ij}$. So information exchanged along one edge is a scalar. An extension to projections onto subspaces, with $Y_{ij} \in \R^{n \times m}$, could also be considered. The natural cost function for this setting is 
$$g_o \; = \; \frac{1}{2}\sum_{(i,j) \in \E} \; (x_i^T y_{ij} - x_j^T y_{ij})^2 \; = \; \frac{1}{2}\sum_{(i,j) \in \E} \; (x_i-x_j)^T \, M_{ij} \, (x_i-x_j)$$
with $M_{ij}=y_{ij} y_{ij}^T$ defined as in Section \ref{ss:debut}. The ensuing gradient law writes
\begin{equation}\label{eq:RnCons}
\tfrac{d}{dt} x_i = \sum_{\{j : (i,j) \in \E\}} \, M_{ij} \, (x_j  -x_i) \quad, \quad i=1,2,\ldots,k \, .
\end{equation}
One term of (\ref{eq:RnCons}) drives the components along direction $y_{ij}$ of $x_i$ and $x_j$ towards each other, leaving the rest of the state unchanged.
The idea is that, by combining information along different directions $y_{ij}$ of different edges, synchronization can be achieved on the whole state vector. This setting essentially differs from previous work \cite{Emre08,ScardoviLinearSyst}, where state-to-output map is the same for all edges and state synchronization relies on zero-input dynamics $\tfrac{d}{dt} x_i = A x_i$ making an individual system detectable from its single output.


\section{Convergence for two agents}\label{sec:2agents}

For $k=2$ agents, there is a single edge $(1,2)$ with \refy vector $y_{12}$.
Moreover, invariance w.r.t.~absolute orientation (or position) allows to reduce the state of the system to the relative state of one agent with respect to the other one. Convergence properties on $SO(n)$ and on $\R^n$ are similar: for fixed $y_{12}$, synchronization appears only along the space on which information is projected; with varying $y_{12}$ full state synchronization is possible.\vspace{2mm}

On $SO(n)$, defining $Q = Q_1 Q_2^T$ leads to reduced dynamics
\begin{eqnarray}
\label{alg:2full} \text{\emph{full state:}} \quad \tfrac{d}{dt} Q & = & 2 Q \sk(Q^T)\\
\label{alg:2partial} \text{\emph{partial state:}} \quad \tfrac{d}{dt} Q & = &  2Q \sk(Q^T M_{12}) \, .
\end{eqnarray}
These expressions are also valid for a directed link, that is when e.g.~$Q_1$ follows (\ref{alg:all}) and $Q_2$ remains static. Therefore, (\ref{alg:2full}) or (\ref{alg:2partial}) are also building blocks for tracking / observer algorithms: if an external pilot imposes $\tfrac{d}{dt}Q_2(t)=Q_2\,\Omega(t)$ with known input $\Omega(t) \in \so(n)$, then (\ref{alg:2partial}) remains valid if $\tfrac{d}{dt}Q_1 = Q_1 \, (\, \sk(Q_1^T M_{12}Q_2) + \Omega(t) \,)$, see also \cite{bonnabel09:LieObservers,LTM10,2009:SO3:CDC}.

Dynamics (\ref{alg:2full}) and~(\ref{alg:2partial}) implement gradient descent for cost functions $\hat{f}_s(Q) =  \Vert Q^T  - I_n \Vert_F^2$ and $\hat{f}_o(Q) = \Vert Q^T y_{12} - y_{12} \Vert_2^2$ respectively. In terms of $Q$, state and output synchronization map to $\hat{C}_s = \{ I_n \}$ and $\hat{C}_o = \{ Q : \hat{f}_o(Q) = 0 \,\} = \st(y_{12})$ respectively. For~(\ref{alg:2full}), $Q = I_n$ is the only stable equilibrum, see~\cite{ss:08}. For~(\ref{alg:2partial}) the system behaves as follows.

\begin{proposition}\label{prop:2agents}
Consider the dynamics (\ref{alg:2partial}) on $SO(n)$ with $n>2$. 
\newline (a) Let $Z_y = \{ Q : Q y_{12} = - y_{12}\}$ the set of maxima of $\hat{f}_o$.
If $y_{12}$ is constant, then all trajectories with initial condition $Q(0) \in SO(n) \setminus Z_y$ 
asymptotically converge to $\hat{C}_o$; convergence is locally exponential. 
Only initial points in
$S_0 = 
\{ I_n + (\cos(\theta)$-$1)\, (y_{12} y_{12}^T \! + \! z z^T) + \sin(\theta) \, (y_{12} z^T \! - \! z y_{12}^T) \, : \, z \in S^{n-1},\, z^T y_{12} = 0,\; \theta \in (-\pi,\pi) \}$
converge to $\hat{C}_s$.
\newline (b) If $y_{12}(t)$ is a time-varying smooth function, $\tfrac{d}{dt}y_{12}$ is uniformly bounded and there exist $T, \alpha > 0$ 
such that $\tfrac{1}{T}\int_t^{t+T} M_{12}(\tau) d\tau - \alpha I_n$ is positive definite for all $t>0$, 
then $\hat{C}_s$ is the only asymptotically stable limit set and attracts (at least) all solutions starting in $SO(n) \setminus S_1$, where $S_1 = \{Q\in SO(n): \hat f_s(Q) \ge 2 \}$.
\end{proposition}
\begin{proof}
(a) Take $\hat{f}_o$ as Lyapunov function. Along solutions of (\ref{alg:2partial}), 
$\tfrac{d}{dt}\hat f_o = -\Vert \tfrac{d}{dt} Q \Vert^2_F = -(1+y_{12}^T Q y_{12}) \hat f_o \leq 0$. 
$SO(n)$ is compact and all functions are smooth, so by the LaSalle invariance principle (see e.g.~\cite{Khalil}) all solutions converge to the set where $\tfrac{d}{dt}\hat f_o=0$. Points with $1+y_{12}^T Q y_{12} = 0$ belong to $Z_y$; since they are maxima of $\hat f_o$, they attract no trajectories. So all trajectories starting outside $Z_y$ converge to $\hat f_o = 0$. In the neighborhood of $\st(y_{12})$, we have $y_{12}^T Q y_{12} > 0$ thus $\tfrac{d}{dt}\hat f_o < - \hat f_o$, which yields exponential convergence.

A trajectory $Q(t)$ of (\ref{alg:2partial}) starting at $Q(0) \in S_0$ stays in $S_0$ for all $t>0$; writing $Q$ as in the definition of $S_0$, dynamics reduce to
$\tfrac{d}{dt}z=0$, $\tfrac{d}{dt}\theta = -\tfrac{1}{2}\sin(\theta)$. Thus $\theta$ converges to $0$ for $\theta(0) \neq \pi$, so points in $S_0$ converge to $\hat{C}_s = \{ I_n \}$.
Denote $\bar S_0 = \{ I_n + (\cos(\theta)$-$1)\, (y_{12} y_{12}^T \! + \! z z^T) + \sin(\theta) \, (y_{12} z^T \! - \! z y_{12}^T) \, : \, z \in S^{n-1},\, z^T y_{12} = 0,\; \theta \in [-\pi,\pi] \}$. Note that $\bar{S}_0 \setminus S_0 \subset Z_y$. 
Any $Q \in SO(n)$ can be written $Q = R U$ with $R \in \st(y_{12})$ and $U \in \bar{S}_0$; this can be checked e.g.~by expressing $Q$ in an orthonormal basis with first element $y_{12}$ and showing that there is at least one solution $(R,U)$ to $Q=R U$ for any $Q\in SO(n)$.
Now if $U^*(t)$ is a solution of (\ref{alg:2partial}), 
then $R U^*(t)$ is also a solution of (\ref{alg:2partial}), for any $R \in \st(y_{12})$. 
Therefore solutions starting at $Q(0) = R U(0)$
with $U(0) \in S_0$ converge to $R$; solutions starting at $Q(0) = R U(0)$ with $U(0) \in \bar S_0\setminus S_0 \subset Z_y$ are constant.
Thus solutions starting outside $S_0$ do not converge to $\hat{C}_s = \{ I_n \}$.

(b) Take $\hat{f}_s$ as Lyapunov function: 
$\tfrac{d}{dt}\hat f_s = -2\tr(\tfrac{dQ}{dt}) = 2((Q y_{12})^T (Q^T y_{12}) -1) \leq 0$ along solutions of (\ref{alg:2partial}), with equality at time $t$ iff $Q^2(t) \in \st(y_{12}(t))$. Points outside $\{I_n\}$ cannot be stable since $I_n$ is the only local minimum of $\hat f_s$, which never increases along trajectories.

Since $y_{12}(t)$ is time-varying, a modified version of the LaSalle invariance principle must be used \cite{Aeyels1995}.
For a given set $S_1$, 
it ensures that $Q(0) \in SO(n)\setminus S_1$ asymptotically converge to $\{ I_n\}$ if, in addition to standard LaSalle conditions and technical requirements that automatically hold thanks to compactness of $SO(n)$ and smoothness of (\ref{alg:2partial}), the following hold: 
(i) solutions starting in $SO(n)\setminus S_1$ at $t=0$ remain in $SO(n)\setminus S_1$ for all $t>0$; 
(ii) denoting by $Q(t+\news; \bar{Q}, t)$ the solution at time $t+\news$ of (\ref{alg:2partial}) with initial condition $Q(t) = \bar{Q}$, there exists for any $\bar{Q} \in SO(n)\setminus (S_1 \cup \{I_n\})$ a finite time $\tau(\bar{Q})$ such that $\limsup_{t \rightarrow \infty} \hat f_s(Q(t+\tau(\bar{Q}); \bar{Q}, t),\, t+\tau(\bar{Q})) < \hat f_s(\bar{Q},\, t)$.

Taking $S_1= \{Q : \hat f_s(Q) \ge 2 \}$, (i) holds since $\tfrac{d}{dt}\hat f_s \leq 0$ along trajectories. We now show that (ii) holds for $\tau(\bar Q)= T$. 
Assume that this is not the case, i.e. there is a $\bar{Q}\in SO(n)\setminus S_1$ such that
\begin{equation}\label{eq:lims}
\limsup_{t \rightarrow \infty} \hat f_s(Q(t+\news; \bar{Q}, t)) = \hat f_s(\bar{Q}) \quad \forall \news \in [0,T] \; .
\end{equation}
Since $\tfrac{d}{dt}y_{12}$ is uniformly bounded, $y_{12}(t)$ is uniformly continuous.
Therefore $\tfrac{d}{d\news}\hat{f}_s(Q(t+\news; \bar{Q}, t))$ is uniformly continuous (in $\news$ and $t$) and it is necessary for satisfying (\ref{eq:lims}) that
\begin{eqnarray}
\nonumber 0 & = & \limsup_{t \rightarrow \infty} \, \min_{\news \in [0,T]} \, \tfrac{d}{d\news}\hat{f}_s(Q(t+\news; \bar{Q}, t)) \\
\label{eq:toref} & = & \limsup_{t \rightarrow \infty} \, \min_{\news \in [0,T]} \, 2\left( \, y_{12}(t+\news)^T \, (Q(t+\news; \bar{Q}, t))^2\, y_{12}(t+\news) \;-1 \, \right) \, .
\end{eqnarray}
This requires the existence, for any given $\eps > 0$, 
of a sequence $t_m \rightarrow \infty$ for which $1- y_{12}(t_m+\news)^T (Q(t_m+\news; \bar{Q}, t_m)^T)^2 y_{12}(t_m+\news) < \eps$ for all $\news\in[0,T]$.
Note that all $Q\in SO(n)$ with eigenvalues $e^{i \phi_j}$, $|\phi_j|<\pi/2$, satisfy $y^T Q^2 y \le y^T Q y$ for all $y\in S^{n-1}$; this is easily seen using the Jordan normal form of $Q$. Since $\bar{Q} \notin S_1$, there exists $\beta(\bar{Q}) > 0$ such that $\hat f_s(Q) = 2 (n - \tr(Q)) = 2 (n - \sum_{j=1}^n \text{real}(e^{i \phi_j})) \leq 2-2\beta$, equivalently $\sum_{j=1}^n \cos(\phi_j) \geq n-1+\beta$; this requires $|\phi_j| \leq \pi/2-\beta$ for all $j$. Therefore, 
$$
\eps > y_{12}(t_m+\news)^T (I_n - Q^2(t_m+\news; \bar{Q}, t_m)) y_{12}(t_m+\news) \ge 1-(y_{12}(t_m+\news)^T Q(t_m+\news; \bar{Q}, t_m) y_{12}(t_m+\news))
$$
for all $\news\in [0,T]$ and $m\in\N$. Bounding $\Vert Q(t+\news;\bar{Q}, t) - \bar{Q} \Vert_{F}$ 
by $\int_0^{\news} \Vert \tfrac{d}{dl} Q(l;\bar{Q},t)\vert_{l=t+\news} \Vert_{F} \, dl$ and using standard norm properties, we then have 
\begin{eqnarray*}
\lefteqn{\int_0^T \, 1-\tr(M_{12}(t_m+\news) \bar{Q}) d\news} && \\
	& \leq & \left\vert \int_0^T \, 1-\tr(M_{12}(t_m+\news) Q(t_m+\news;\bar{Q}, t_m) ) \, d\news \right\vert  \\
	&& \;\;\; + \left\vert \int_0^{T} y_{12}^T(t_m+\news) (Q(t_m+\news;\bar{Q}, t_m)-\bar{Q}) y_{12}(t_m+\news) d\news \right\vert \\
	& \leq & \int_0^T \, (\, 1 - \tr(M_{12}(t_m+\news) Q(t_m+\news;\bar{Q}, t_m) ) \,) \, d\news 
	+ \int_0^T \, \Vert Q(t_m+\news;\bar{Q}, t_m) - \bar{Q} \Vert_{F} \, d\news \\
	& \leq & T \, \max_{\news \in [0,T]} (\, 1 - \tr(M_{12}(t_m+\news) Q(t_m+\news;\bar{Q}, t_m) ) \,) 
	+ \tfrac{T^2}{2} \, \max_{\news \in [0,T]} \left( \Vert \tfrac{d}{dl} Q(l;\bar{Q},t_m)\vert_{l=t_m+\news} \Vert_{F} \right) \\
	& \leq & \max_{\news \in [0,T]} \, (T A + T^2 \sqrt{A}) \qquad \text{with }\;\; A=1-(y_{12}(t_m+\news)^T Q(t_m+\news;\bar{Q}, t_m) y_{12}(t_m+\news)) \\
	& \leq & T \eps + T^2 \sqrt{\eps} \, .
\end{eqnarray*}
Thus we need
$\; 0 = \liminf_{t\rightarrow\infty} \tfrac{1}{T} \int_0^T \, 1-\tr(M_{12}(t+\news) \bar{Q}) d\news = \liminf_{t\rightarrow\infty} \tr(B(t) (I_n- \bar{Q})) \;$ 
with $B(t) := \tfrac{1}{T} \int_0^T M_{12}(t+\news) \, d\news$ symmetric 
and $B(t)-\alpha I_n$ positive definite for all $t>0$. 
This is only possible if $\bar{Q}=I_n$. 
Therefore (ii) will be satisfied.
\end{proof}

\begin{remark}
(i) It is still unclear whether \emph{exponential} convergence to $\hat{C}_s$ holds in case (b). \newline
(ii) The decrease of $\hat{f}_s \cong f_s$ along solutions is a particularity of $k=2$.
\end{remark}

We briefly consider the $\R^n$ case. 
Defining $x = x_2 - x_1$ leads to $\tfrac{d}{dt}x = -2x$ (full-state coupling) or $\tfrac{d}{dt}x = - 2 M_{12} x$ (partial-state coupling). The first system obviously converges to $x=0$ $\Leftrightarrow$ $x_1 = x_2$. 
Convergence with partial-state coupling is established as follows.
\begin{itemize}
\item For $y_{12}$ fixed, $x$ globally converges to $(I_n - y_{12} y_{12}^T)\, x$, 
i.e.~the projection of $x_2-x_1$ onto $\operatorname{span}\{y_{12}\}$  is driven to $0$ and the projection of $x_1-x_2$
onto the orthogonal complement of $\operatorname{span}\{y_{12}\}$ remains constant.
\item For varying $y_{12}$ and under the same conditions as for $SO(n)$, $x$ globally converges to $0$, implying state synchronization. This standard result actually holds under weaker conditions, see \cite{ASP98}. It is currently not clear how to adapt \cite{ASP98} to $SO(n)$. This setting is equivalent to \cite{Emre08} in the special case where $\tfrac{d}{dt}y_{12} = \Omega \, y_{12}$ here, and zero-input individual system dynamics are $\tfrac{d}{dt}x_i = -\Omega \, x_i$ in \cite{Emre08}, with $\Omega \in \so(n)$.
\end{itemize}


\section{Synchronization conditions for $k>2$ agents}\label{sec:fixed}

We now turn to the general case of more than two agents.
We provide an extensive analysis for fixed $y_{ij}$ in Section~\ref{ssec:fix} and a
result for time-varying $y_{ij}$, similar to Proposition~\ref{prop:2agents}~(b), in Section~\ref{ssec:timvar}.

\subsection{Fixed \refy vectors}\label{ssec:fix}

We start with the case of fixed \refy vectors. An immediate result of the construction of our gradient law is the stability of the output synchronization set.

\begin{proposition}\label{prop:ocstable}
The output synchronization set is asymptotically stable under (\ref{alg:all}).
\end{proposition}
\begin{proof}
Since (\ref{alg:all}) is a gradient descent system for $f_o$ and $f_o$ is an analytic function, the set
$C_o$ of global minima of $f_o$ is asymptotically stable~\cite{ak:06}. 
\end{proof}

The same clearly holds for  (\ref{eq:RnCons}), the corresponding system in $\R^n$.
However on $SO(n)$, as will be illustrated in Section~\ref{sec:SO3},
\emph{convergence is not necessarily exponential} for $k>2$. 

\subsubsection{Necessary condition for state synchronization}

The following theorem formalizes the natural condition: to reach state synchronization, each agent's state must be observable by combining information on all edges incident to the agent.
Later we show that this is \emph{not} sufficient (at all) for $C_s = C_o$.
The necessary condition for $C_s=C_o$ in Theorem~\ref{thm:inj} can be extended in straightforward way to the following necessary condition: 
considering any partition of $\V$ into $\V_1$ and $\V_2$, the information on edges connecting $\V_1$ to $\V_2$ must make a full state $\in SO(n)$ observable between the two subsets, i.e. there must be at least $n-1$ links on the cut with linear independent \refy vectors.

\begin{theorem} \label{thm:inj}
The \emph{output map} from state space $(SO(n))^k$ (resp.~$(\R^n)^k$) to output space $(S^{n-1})^{2\#\E}$ (resp.~$\R^{2\#\E}$) is injective if and only if each $i\in \V$ has at least $n-1$ (resp.~$n$) linearly independent $y_{ij}$ with $a_{ij}=1$.
If injectivity is not satisfied, then for any $(Q_1,\ldots,Q_k)$ in $C_o$ there is a continuum of states in $C_o\setminus C_s$ giving the same output as $(Q_1,\ldots,Q_k)$.
\end{theorem}
\begin{proof}
We first discuss the $SO(n)$ case. Consider an agent $i$ and number the agents such that $a_{ij}=1$ if and only if $j \le m$.
Denote by $h_i\colon SO(n)\times (S^{n-1})^m \rightarrow (S^{n-1})^m$
the action of $Q_i$ on $\hat{y}_i := (y_{i 1},\ldots, y_{i m})$, 
i.e. $h_i(Q_i,\hat{y}_i)=(Q_i^T y_{i 1},\ldots, Q_i^T y_{i m})$. 
The map $Q_i\mapsto h_i(Q_i,\hat y_{i})$ is the output map of agent $i$. Define
$\st_h(\hat{y}_i):=\{ Q\in SO(n): h_i(Q,\hat{y}_i)=\hat{y}_i \} = \bigcap_{j=1}^m \st(y_{ij})$.
Thus by definition $y_{i1},\ldots,y_{im}$ are eigenvectors with eigenvalue $1$
of all $Q\in \st_h(\hat{y}_i) \subseteq SO(n)$.
Then $\st_h(\hat{y}_i)=\{I_n\}$ if and only if $\{y_{i1},\ldots,y_{im}\}$ contains $(n-1)$ linearly independent vectors. 
Indeed, the condition implies that all $Q \in \st_h(\hat{y}_i)$ have $(n-1)$ eigenvalues $1$, and $Q \in SO(n)$ implies that the remaining eigenvalue also equals $1$; conversely, if there is a $\geq2$-dimensional subspace $S_0$ orthogonal to $\{y_{i1},\ldots,y_{im}\}$, then all $Q$ representing rotations in $S_0$ belong to $\st_h(\hat{y}_i)$.
By invariance of linear independence under a common rotation, the same holds if $\hat{y}_i$ is replaced by $h_i(Q^*,\hat{y}_i)$ for any $Q^* \in SO(n)$.
Therefore a particular output $h_i(Q^*_i,\hat{y}_i)$ of agent $i$ corresponds to a unique state $Q^*_i$ of agent $i$ if and only
if $\{y_{i1},\ldots,y_{im}\}$ contains $(n\!-\!1)$ linearly independent vectors.
As the whole output map is just the collection of $h_i$ for all $i \in \V$, with $h_i$ depending on the state of agent $i$ only, 
this proves the first claim.
As said above, if the condition does not hold for agent $i$ then $\st_h(\hat{y}_i)$ contains rotations in a subspace $S_0$, that form a Lie subgroup of $SO(n)$ of dimension $\ge 1$. If $(Q_1,\ldots,Q_k) \in C_o$, then $(Q_1,\ldots,\hat{Q}Q_i,\ldots,Q_k) \in C_o$ yields the same output for all $\hat{Q}\in \st(\hat{y}_i)$, but at most one $\hat{Q}$ corresponds to state synchronization.

The analog property for $\R^n$ is straightforward: projections on $n$ linearly independent \refy vectors specify the state completely, while less than $n$ projections leave the possibility to move the state in one direction without affecting the output.
\end{proof}\vspace{4mm}

The following Examples~\ref{ex:nocon} and~\ref{ex:Rn} illustrate that injectivity of the output map \emph{does not guarantee $C_s=C_o$}. This happens because there exist output values $Q^T_i y_{ij}$ for states in $C_o$, that cannot be obtained with states in $C_s$. Therefore Theorem~\ref{thm:inj} only gives a necessary condition: unlike what first intuition might suggest, \emph{each agent injecting full state information into the network does not imply state synchronization.}
This suggests that a theory of observability in networks is non-trivial.

\begin{example} \label{ex:nocon}
Consider 3 agents on $SO(3)$ with a fully connected interaction graph and $y_{12} = (1,0,0)^T$, $y_{23} = (0,1,0)^T$, $y_{13}=(0,0,1)^T$. Assumptions of Theorem~\ref{thm:inj} hold, so each output synchronization value corresponds to a unique point in state space. Let output synchronization be reached at $Q_1^T y_{12} = Q_2^T y_{12} = (1,0,0)^T$, $Q_1^T y_{13} = Q_3^T y_{13} = (0,0,1)^T$ and $Q_2^T y_{23} = Q_3^T y_{23} = (0,-1,0)^T$. The corresponding unique state is
$Q_1=\diag(1,1,1) \, , \; Q_2=\diag(1,-1,-1) \, , \; Q_3=\diag(-1,-1,1)$
which is not state synchronization.
\end{example}

\begin{example} \label{ex:Rn}
Consider 4 agents with states $x_i \in \R^2$ where $\E = \{\,(1,2),\,(1,4),\,(2,3),\,(3,4)\,\}$ and with $y_{12}= (1, 0)^T$, $y_{14}=(0, 1)^T$, $y_{23}=(\cos\theta,\sin\theta)^T$, $y_{34}=(\cos\phi,\sin\phi)^T$, assuming $\theta,\phi \not\in\{k \pi/2 : k \in\mathbb{Z} \}$ and $y_{23} \neq \pm y_{34}$. The setting is illustrated in Fig.~\ref{fig:RnEx}. Assumptions of Theorem~\ref{thm:inj} hold, so the output map is injective. Note that the generalized necessary condition for $C_s=C_o$ discussed before Theorem~\ref{thm:inj} is also satisfied. Nevertheless, states of the form $x_1=(a,b)^T$, $x_2=(a,\,d+\tfrac{c-a}{\tan\theta})^T$, $x_3=(c,d)^T$, $x_4=(c+(d-b)\tan\phi,\,b)^T$ correspond to output synchronization for any $(a,b,c,d) \in \R^4$, while state synchronization only holds if $c=a$ and $d=b$.
\end{example}

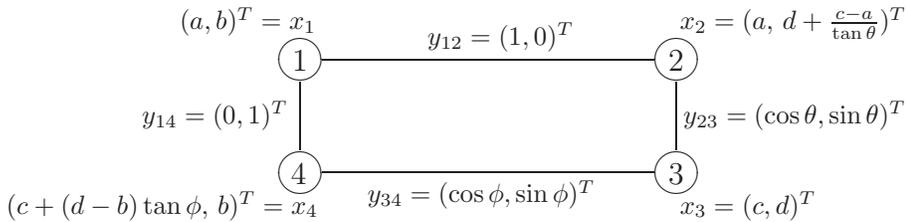
\begin{figure}[hbt]
	\begin{center}
		\setlength{\unitlength}{1mm}
		\begin{picture}(60,23)(0,2)
			\multiput(3,5.5)(0,15){2}{\circle{6}}
			\put(1.8,4){$4$} \put(2,19){$1$}
			\multiput(53,5.5)(0,15){2}{\circle{6}}
			\put(52,4){$3$} \put(52,19){$2$}
			\multiput(6,5.5)(0,15){2}{\line(1,0){44}}
			\multiput(3,8.5)(50,0){2}{\line(0,1){9}}
			\put(20,22.5){\footnotesize $y_{12}=(1,0)^T$}
			\put(12,1.5){\footnotesize $y_{34}=(\cos\phi,\sin\phi)^T$}
			\put(-18,12){\footnotesize $y_{14}=(0,1)^T$}
			\put(54,12){\footnotesize $y_{23}=(\cos\theta,\sin\theta)^T$}
			\put(-13,24.5){\footnotesize $(a,b)^T=x_1$}
			\put(53.5,24.5){\footnotesize $x_2=(a,\,d+\tfrac{c-a}{\tan\theta})^T$}
			\put(53.5,0){\footnotesize $x_3=(c,d)^T$}
			\put(-36,0){\footnotesize $(c+(d-b)\tan\phi,\,b)^T=x_4$}
		\end{picture}
	\end{center}
	\caption{
	Example~\ref{ex:Rn} --	each agent sharing its whole state with the network is not sufficient to ensure that output synchronization implies state synchronization.}\label{fig:RnEx}
\end{figure}


\subsubsection{Sufficient conditions for state synchronization}

The output cost functions can be written as a generalization of the Laplacian-based quadratic form often used in the context of synchronization algorithms: 
\begin{eqnarray}\label{rew:SOn}
f_o(Q_1,\ldots,Q_k) & = &
\frac{1}{2} \; \left( \vect \begin{pmatrix}
Q_1 \\
\vdots \\
Q_k
\end{pmatrix} \right)^T \; (I_n \otimes L^g) \; \;
\left(\vect \begin{pmatrix}
Q_1 \\
\vdots \\
Q_k
\end{pmatrix} \right) \\[2mm]
\label{rew:Rn}
g_o(x_1,\ldots,x_k) & = & \frac{1}{2} \; \begin{pmatrix}
x_1 \\
\vdots \\
x_k
\end{pmatrix}^T \, L^g \, \begin{pmatrix}
x_1 \\
\vdots \\
x_k
\end{pmatrix}
\end{eqnarray}
where $L^g = ( L^g_{ij} ) \in \R^{kn\times kn}$ is an extended Laplacian matrix with $L^g_{ij}\in \R^{n\times n}$ defined by \newline $
L^g_{ij} = -a_{ij} M_{ij} \quad \text{  for } i\not=j \, ,\qquad
L^g_{ii} = \sum_{j=1}^k a_{ij} M_{ij}\; .$
The standard Laplacian matrix $L$ has the same definition but with $M_{ij}$ replaced by $1$, i.e.~$L_{ij} \in \R$. For undirected graphs, $L= B B^T$ where $B \in \R^{k \times \#\E}$ is the \emph{incidence matrix} of $\G$: each column corresponds to an edge $e$ and contains all zeros except $1$ on row $\vr(e)$ and $-1$ on row $\vl(e)$. It is straightforward to check that $L^g = (B \otimes I_n) \, W \, (B \otimes I_n)^T$ where $W \in \R^{n(\#\E) \times n(\#\E)}$ is block diagonal with $M_{ij}$ in the block corresponding to edge $(i,j)$. Standard rank properties yield
\begin{equation}\label{eq:rankprop}
	\rk(L^g) \leq \min(\rk W ,\; n\,\rk B ) =  \min(\# \E,\; n\,\rk L ) \, .
\end{equation}
The rank of $L^g$ is thus limited by number of edges and graph connectivity:
for a graph on $k$ vertices with $m$ connected components, $\rk L  = k-m$. However (\ref{eq:rankprop}) is far from being tight, as it contains no information about the output map while $L^g$ does: if e.g.~all $y_{ij}$ are equal, then $\rk L^g = \rk L \le (k-1)$.
For agents in $\R^n$, $L^g$ taking its maximal possible rank is necessary and sufficient for global, exponential state synchronization.

\begin{proposition} \label{thm:LgRn}
$C_o = C_s$ if and only if $\rk L^g = n(k-1)$ for partial state coupling on $\R^n$.
Moreover, if $\rk L^g =n(k-1)$ the set $C_o=C_s$ is globally exponentially stable under (\ref{eq:RnCons}).
\end{proposition}
\begin{proof}
Rewriting (\ref{eq:RnCons}) using (\ref{rew:Rn}) we get $\tfrac{d}{dt}x = -L^g \, x$ with $x = (x_1^T, \ldots, x_k^T)^T \in \R^{nk}$. As $L^g$ is positive semidefinite, $x$ converges to $\ker(L^g) = C_o$; the latter is stable by Proposition \ref{prop:ocstable}. 
Since $C_s \subseteq C_o = \ker(L^g)$ and $C_s$ has dimension $n$, we have $C_o = C_s$ if and only if $\rk L^g = n(k-1)$. Linearity implies that convergence is exponential and global.
\end{proof}

The situation on $SO(n)$ is quite different: (i) the rank condition is only sufficient; (ii) state synchronization is usually not global, even for a complete graph satisfying the rank condition; and (iii) convergence is not necessarily exponential.

\begin{theorem} \label{thm:LgSOn}
Consider dynamics (\ref{alg:all}) for partial state coupling on $SO(n)$.\newline
(a) If $\rk L^g = n(k-1)$ then $C_o = C_s$ and the set $C_o=C_s$ is locally exponentially stable.\newline
(b) Denote $V \in \R^{n^2 k \, \times \, kn(n-1)/2}$ a matrix whose column vectors are an orthonormal basis for $\vect(\so (n)^k) \subseteq \mathbb{R}^{n^2k}$. $C_s$ is locally exponentially stable if and only  if $\; \rk (V^T (I_n \otimes L^{g}) V) = (k-1) \tfrac{n(n-1)}{2}$.
\newline NB1: Non-exponential asymptotic stability of $C_s$ is still possible if the rank condition fails.
\newline NB2: We refer to Section~\ref{sec:notation} for the definition of $\vect(\so (n)^k)$.
\end{theorem}
\begin{proof}
Extend $f_o(X_1,\ldots,X_k)$ to $X_i \in \R^{n\times n}$. 
Denote $S$ the subspace of global minima of $f_o$ on $(\R^{n \times n})^k$. 
Obviously $C_o \subset S$. 
In addition, $S$ contains the subspace
$C_s^r := \{(X, \ldots, X) :\, X \in \R^{n \times n} \}$. If $\rk(L^g) = n (k-1)$ then $\rk(I_n\otimes L^g) = n^2k - n^2$ and from (\ref{rew:SOn}), $S$ is $n^2$-dimensional, like $C_s^r$, thus $S = C_s^r$. Then $C_o \subset S = C_s^r$ and the intersection with $(SO(n))^k$ yields $C_o = C_s$.

We apply Lemma~\ref{lem:BasicConv} from the appendix. 
For $(\hat{Q},\ldots,\hat{Q}) \in C_s$, 
we have $L^{gg} = (I_{k}\otimes \hat{Q}^T) L^g (I_{k}\otimes \hat{Q})$, with $L^{gg}$ as defined in Lemma~\ref{lem:BasicConv}. Thanks to right-invariance w.r.t.~absolute orientation, 
we can reduce the investigation to $\hat{Q} = I_n$ without loss of generality, so $L^{gg} = L^g$.
From the previous paragraph, if $\rk L^g = n(k-1)$ then $I_n \otimes L^g$ is positive definite on the subspace of $\R^{kn^2}$ orthogonal to $\vect(C_s^r)$, 
thus in particular
on the subspace of $\vect(\so (n)^k)$ orthogonal to $\vect(C_s^r)$.
So the condition of Lemma~\ref{lem:BasicConv} is satisfied and we get exponential stability of $C_o$, proving (a). By definition, $L_V := V^T (I_n \otimes L^{g}) V$ is the restriction of $(I_n \otimes L^{g})$ to $\vect(\so (n)^k)$. Obviously, $\vect(C_s^r) \cap \vect(\so(n)^k) = \{\vect((\Omega^T,\ldots, \Omega^T)^T) : \Omega \in \so(n) \} \subseteq \ker(L_V)$. Since the dimension of $\vect(C_s^r) \cap \vect(\so (n)^k)$ is $\tfrac{n (n-1)}{2}$, it follows that $L_V$ is full rank in the subspace of $\vect(\so (n)^k)$ orthogonal to $\vect(C_s^r)$ if and only if $L_V$ has rank $\tfrac{kn(n-1)}{2} - \tfrac{n (n-1)}{2} = (k-1) \tfrac{n(n-1)}{2}$, proving (b).
\end{proof}\vspace{3mm}

Conditions (a) and (b) of Theorem~\ref{thm:LgSOn} are complementary. 
In condition (a), roles of graph connectivity and \refy vectors can be visualized to some extent. 
In particular, from~(\ref{eq:rankprop}), $\rk L^g = n(k-1)$ requires at least $n(k-1)$ edges and thus $k \ge 2n$ agents. 
Moreover, checking (a) involves $L^g \in \R^{nk \times nk}$ and is less costly than (b) which involves $n^2 k \times n^2 k$ matrices.
On the other hand, (b) examines the rank of $(I_n \otimes L^g)$ on the tangent space of $(SO(n))^k$ and therefore is tighter than (a) which basically checks it in $\R^{n^2 k \times n^2 k}$. Condition (b) can even prove local exponential stability of $C_s$ in situations where $C_o \neq C_s$. If condition (b) fails, $C_s$ can still be asymptotically stable, although convergence cannot be exponential. Section~\ref{sec:SO3} contains examples illustrating these properties for $SO(3)$.\vspace{2mm}

\begin{example}\label{Ex:SO2}
On $SO(2) \equiv S^1$, an output $Q_i^T y_{ij}$ is equivalent to full state knowledge: explicitly writing (\ref{alg:all}) and (\ref{FullDyns}) in terms of rotation angles yields the same result up to gain factor $2$. It is in fact equivalent to Kuramoto-like dynamics $\tfrac{d}{dt}\theta_i = \sum a_{ij} \; \sin(\theta_j - \theta_i)$ for which local (sometimes almost-global) state synchronization holds as soon as $\G$ is connected, see e.g.~\cite{ss:08}. There are many synchronizing situations where, for Theorem~\ref{thm:LgSOn}(a), $\rk(L^g) < 2(k-1)$, e.g.~when $\G$ has less than $2(k-1)$ edges. In contrast, Theorem~\ref{thm:LgSOn}(b) perfectly captures the situation on $SO(2)$: 
calculations with $V$ constructed e.g.~from basis elements $(\ldots,\left(\begin{smallmatrix} 0 & 0 \\ 0 & 0 \end{smallmatrix}\right) ,\left( \begin{smallmatrix} 0 & 1 \\ -1 & 0 \end{smallmatrix}\right),\left(\begin{smallmatrix} 0 & 0 \\ 0 & 0 \end{smallmatrix}\right),\ldots)$ of $\so(2)^k$ show that actually $V^T (I_n \otimes L^{g}) V = L$.
\end{example}


\subsubsection{About almost-global convergence}

On $\R^n$, stability of state synchronization is always global, see Proposition~\ref{thm:LgRn}. 
On $SO(n)$ with full state exchange, see~\cite{ss:08}, $C_s$ is not always (almost-)globally asymptotically stable, but (\ref{FullDyns}) has no other stable equilibria than $C_s$ when $\G$ is e.g.~all-to-all or a tree. 
For partial state coupling on $SO(n)$, $n > 2$, (\ref{alg:all}) does not retain these global properties\footnote{Contrary to an erroneous statement in \cite{myCDC}, Theorem 2(b).} for trees and complete graphs. For a tree, the number of links is insufficient to share all state information: the necessary condition from Theorem~\ref{thm:inj} is not satisfied. 
The following example illustrates that even with all-to-all interaction and conditions of Theorem~\ref{thm:LgSOn} holding, stable equilibria
besides the state consensus set can exist.

\begin{example}\label{Ex:NotGlobal}
Consider $k=7$ agents on $SO(3)$, interconnected with: $y_{ij} = (0,\, 1,\, 0)^T$ if $(i-j) \, \mathrm{mod}\, 7 = \pm 1$, $y_{ij} = (0,\, 0,\, 1)^T$ if $(i-j) \, \mathrm{mod}\, 7 = \pm 3$ and $y_{ij} = (\sin(\alpha) ,\, 0,\, \cos(\alpha))^T$ if $(i-j) \, \mathrm{mod}\, 7 = \pm 2$, with $\alpha = 0.04$.
This implies $\rk(L^g) = 18 = n(k-1)$, so Theorem~\ref{thm:LgSOn} applies and $C_s=C_o$ is locally stable.
Nevertheless, (\ref{alg:all}) admits an equilibrium $\notin C_s$ of the form $Q_i = Q_y \, Q_z(i \tfrac{2 \pi}{7})$, $i = 1,\ldots,7$, with
$$
Q_z(\theta) = \left( \begin{array}{ccc}
\cos(\theta) & \sin(\theta) & 0 \\ -\sin(\theta) & \cos(\theta) & 0 \\ 0 & 0 & 1
\end{array}\right)\qquad \text{,} \qquad Q_y = \left( \begin{array}{ccc}
\cos(\beta) & 0 & \sin(\beta) \\ 0 & 1 & 0 \\ -\sin(\beta) & 0 & \cos(\beta)
\end{array}\right)
$$
and $\beta = -0.0558$.
Checking the Hessian at that point shows that it is stable. The small but nonzero values of $\alpha$ and consequently $\beta$, are necessary to have at the same time $L^g$ of maximal rank and stability of the equilibrium.
\end{example}


\subsection{Time-varying \refy vectors}\label{ssec:timvar}

Time-varying $y_{ij}$ are favorable for state synchronization, as they give access to different partial state information over time. We assume $y_{ij}$ smooth functions of time $\R\rightarrow S^{n-1}$.
Unlike for $k=2$ agents, $f_s$ does not decrease along solutions of~(\ref{alg:all}) in general.
This somewhat complicates the analysis and forces us to introduce a scaling parameter $\varepsilon$ in \eqref{alg:all}, i.e.
\begin{equation}\label{eq:tv}
\tfrac{d}{dt} Q_i =  \varepsilon Q_i \sum_{\{j : (i,j) \in \E\}} \, \sk(Q_i^T M_{ij}(t)\, Q_j) \quad , \quad i=1,2,\ldots,k \, .
\end{equation}
To show convergence to $C_s$, we need a persistent excitation condition on the outputs.
\begin{assumption}\label{assu:1}
$\overline{M}_{ij} := \lim_{t\rightarrow \infty} \; \frac{1}{2t} \int_{-t}^t \; M_{ij}(s) \, ds \;$ exists and is positive definite $\forall (i,j) \in \E$. 
Moreover\footnote{This condition of uniform convergence to $\overline{M}_{ij}$ was not acknowledged in \cite{myCDC}.}, 
there exists $Z : [0,+\infty) \rightarrow [0,+\infty)$ with $\lim_{t\rightarrow \infty} \; \tfrac{1}{t} Z(t) = 0$ 
and such that for all $t$ and $\forall (i,j) \in \E$, it holds $\Vert \int_{t}^{t+T} (M_{ij}(s) - \overline{M}_{ij}) \, ds \Vert_F \leq Z(T)$
for all $T\ge 0$.
\end{assumption}
Assumption~\ref{assu:1} ensures a well-defined ``average system'' (for $\eps=1$)
\begin{equation}\label{eq:av}
\tfrac{d}{dt} \overline{Q}_i = \overline{Q}_i \; \sum_{\{j : (i,j) \in \E\}} \; \sk(\overline{Q}_i^T \overline{M}_{ij} \overline{Q}_j)
\end{equation}
which converges like full-state-coupling case (\ref{FullDyns}), see \cite{ss:08}. For sufficiently small $\varepsilon$,
an averaging argument shows that
this convergence extends to time-varying dynamics~\eqref{eq:tv}.
\begin{proposition}\label{prop:av}
If $\G$ is connected and Assumption~\ref{assu:1} holds, then $C_s$ is locally exponentially stable for \eqref{eq:av}.
\end{proposition}
\begin{proof}
\eqref{eq:av} is a gradient system for the average cost
$\overline{f}(Q_1,\ldots,Q_k)  = 
\tfrac{1}{2}\sum_{(i,j)\in\E}^k  \|\overline{M}_{ij}(Q_i-Q_j)\|^2_F.$
A straightforward extension of the arguments in \cite{ss:08} yields that $C_s$ is locally asymptotically stable, because $C_s$ is an isolated minimum set of cost function $\overline{f}$. Exponential stability holds similarly to other results in the paper.
\end{proof}

\begin{theorem} \label{thm:tv}
If $\G$ is connected and Assumption~\ref{assu:1} holds, then for sufficiently small $\varepsilon>0$, $C_s$ is locally exponentially stable for \eqref{eq:tv}.
\end{theorem}
\begin{proof} 
We first project the system onto one where $C_s$ is represented by a single point. 
For this we use reductive homogeneous spaces, see e.g.~\cite{ce}. 

$C_s$ is a closed subgroup of $SO(n)^k$.
Define the compact homogeneous space $\M = SO(n)^k / C_s$, 
with $C_s$
acting on $SO(n)^k$ by right multiplication, and canonical projection $\pi\colon SO(n)^k\rightarrow \M$. The image of $C_s$ under $\pi$ is a point $p \in \M$.
Since $SO(n)^k$ is compact, $\M$ is a reductive homogeneous space.
Equip $\M$ with the normal Riemannian metric induced by the product metric on $SO(n)^k$.
Denote $F_{SO}(Q_1,\ldots,Q_k,t)$ and $\overline{F}_{SO}(Q_1,\ldots,Q_k)$ respectively the vector fields corresponding to \eqref{eq:tv} with $\varepsilon = 1$ and to \eqref{eq:av}.
Thanks to invariance of the dynamics under right multiplication of $SO(n)^k$ by $C_s$,
they respectively induce vector fields $F(x,t)$ and $\overline{F}(x)$ on $\M$, where $\overline{F}$ is the time average of $F$. Moreover, for both time-varying and time-invariant dynamics, point $p \in \M$ has the same stability properties as set $C_s \subset SO(n)^k$. So $p$ is exponentially stable under $\overline{F}$ and we must prove that this remains true under $\varepsilon F(x,t)$ for $\varepsilon$ small.
We therefore apply Theorem 3 in \cite{ap:99}.

Since all necessary regularity conditions for Theorem 3 of \cite{ap:99} hold,
it remains to examine $\Vert \int_{t_1}^{t_2} F(x,t)-\overline{F}(x) \, dt \Vert$ for any fixed $x \in \M$ in a neighborhood of $p$.
$$\left\Vert \int_{t_1}^{t_2} F(x,t)-\overline{F}(x) \, dt \right\Vert \leq \left\|\int_{t_1}^{t_2} F_{SO}(Q_1,\ldots,Q_k,t)-\overline{F}_{SO}(Q_1,\ldots,Q_k) \, dt \right\| \;$$
by projection from $T_{Q_1,\ldots,Q_k}SO(n)^k$ to $T_x\M$, with $\pi(Q_1,\ldots,Q_k) = x$. 
Now
\begin{eqnarray}
\nonumber \lefteqn{\left\|\int_{t_1}^{t_2} F_{SO}(Q_1,\ldots,Q_k,t)-\overline{F}_{SO}(Q_1,\ldots,Q_k) \, dt \right\|^2} && \\
& = & 
\nonumber \sum_{i=1}^k \left\Vert \sum_{j=1}^k a_{ij} \, \sk(Q_i^T \, \int_{t_1}^{t_2} (M_{ij}(t)-\overline{M}_{ij}) \; dt\, Q_j) \right\Vert_F^2  \\
& \le &
\label{eq:retakeafter1} k^2 \sum_{i=1}^k \sum_{j=1}^k a_{ij} \left\Vert \sk(Q_i^T \, \int_{t_1}^{t_2} (M_{ij}(t)-\overline{M}_{ij}) \; dt\, Q_j) \right\Vert_F^2 \; . 
\end{eqnarray}
Using $\Vert \sk(Q_i^T P Q_j) \Vert_F = \tfrac{1}{2}\Vert P(I_n - Q_jQ_i^T) - (I_n-Q_jQ_i^T)^T P\Vert_F \leq \Vert P \Vert_F \; \Vert I_n - Q_jQ_i^T \Vert_F = \Vert P \Vert_F \; \Vert Q_i - Q_j \Vert_F$ for any symmetric $P \in \R^{n \times n}$, and Assumption~\ref{assu:1}, we get
\begin{eqnarray*}
[\text{Eq.}\eqref{eq:retakeafter1}]	& \leq & 
	 k^2 \sum_{i=1}^k \sum_{j=1}^k a_{ij} \left\Vert\int_{t_1}^{t_2} (M_{ij}(t)-\overline{M}_{ij}) \; dt\right\Vert^2 \left\Vert Q_i-Q_j \right\Vert_F^2\\
	&\le &
	k^2 \; Z(t_2-t_1)^2 \; \sum_{i=1}^k \sum_{j=1}^k a_{ij}\, \Vert Q_i-Q_j \Vert_F^2 \; \leq \; k^2 Z(t_2-t_1)^2 \sum_{i=1}^k \sum_{j=1}^k\, \Vert Q_i-Q_j \Vert_F^2  .
\end{eqnarray*}
Since $\sum_{i,j=1}^k \Vert Q_i - Q_j\Vert_F^2$ is a quadratic form on $(\R^{n\times n})^k$ with kernel
$C_s^r=\{(X,\ldots,X) : X\in\R^{n\times n}\}$, there is a constant $c>0$ such that, with $\operatorname{dist}_{(\R^{n\times n})^k}$ the Euclidean distance,
$\sum_{i,j=1}^k \Vert Q_i - Q_j\Vert_F^2\le c\operatorname{dist}^2_{(\R^{n\times n})^k}((Q_1,\ldots,Q_k),C_s^r)$. 
As $C_s\subset C_s^r$ and we use the Riemannian metric on $SO(n)$ induced by the Euclidean metric on $\R^{n\times n}$, we have
$\operatorname{dist}_{(\R^{n\times n})^k}((Q_1,\ldots,Q_k),C_s^r)\le \operatorname{dist}_{SO(n)^k}((Q_1,\ldots,Q_k),C_s)$ with
$\operatorname{dist}_{SO(n)^k}$ the Riemannian distance on $SO(n)^k$. This yields
$\sum_{i,j=1}^k \Vert Q_i - Q_j\Vert_F^2 \le c\, \operatorname{dist}^2_{SO(n)^k}((Q_1,\ldots,Q_k),C_s) = c\, \operatorname{dist}^2_{\M}(x,p)$
with $p=\pi(C_s)$.
Then overall,
$\; \left\Vert \int_{t_1}^{t_2} F(x,t) - \overline{F}(x)\, dt \right\Vert \leq k\sqrt{c} \, Z(t_2-t_1) \linebreak \operatorname{dist}_{\M}(x,p)\;$.
Using local charts around $p$, this shows that the last condition of Theorem 3 in \cite{ap:99} holds, 
ensuring local exponential stability of $p$ under $F(x,\tfrac{1}{\varepsilon}t)$ for $\varepsilon$ sufficiently small. We conclude by change of timescale: $\tfrac{d}{dt}x = \varepsilon \, F(x,t)$ is equivalent to $\tfrac{d}{d\tau}x = F(x,\tfrac{1}{\varepsilon}\tau)$.
\end{proof}
\vspace{2mm}

It is possible to give quantitative estimates for $\varepsilon$ based on results in \cite{ap:99},
but these technical aspects are not the focus of the present paper.

\begin{remark}
The conditions of Theorem~\ref{thm:tv} can be relaxed to require 
strict positive definiteness of $\overline M_{ij}$, in Assumption~\ref{assu:1}, only on the edges of a spanning tree $\mathcal{G}_s$ of the interaction graph $\mathcal{G}$. Indeed, taking only edges of $\mathcal{G}_{\text{tree}}$, the result of Proposition~\ref{prop:av} directly applies; in particular, $C_s$ is an isolated set of minima for the cost $\overline{f}_{\text{tree}}$ of the average gradient system associated to $\mathcal{G}_{\text{tree}}$ and is locally exponentially stable for that average system. Adding the contribution of edges in $\mathcal{G}\setminus \mathcal{G}_{\text{tree}}$, for which averaged coupling is only partial, $C_s$ remains an isolated minimum set for the cost $\overline{f}$ of the average system associated to $\mathcal{G}$; then it is still locally exponentially stable under the gradient system, so that the conclusion of Proposition~\ref{prop:av} still holds. The proof of Theorem~\ref{thm:tv} can be repeated as such.

It should also be possible to extend the result to networks where a majority of fixed $y_{ij}$ are combined with a few time-varying $y_{ij}$ that enhance the connection on particular links: it suffices to show that $C_s$ is locally exponentially stable for the average system, and then the conclusion of Theorem~\ref{thm:tv} holds.
\end{remark}


\section{State synchronization on $SO(3)$} \label{sec:SO3}

We pursue the analysis for fixed $y_{ij}$ on special case $SO(3)$, which has practical importance as describing rotations in three-dimensional space. It is also the simplest non-trivial output consensus setting.
$Q_i^T y_{ij} = Q_j^T y_{ij}$ fixes $Q_i Q_j^T$ up to a rotation with axis $y_{ij}$, so the coupling on one edge contains $2$-dimensional information about $3$-dimensional relative errors $Q_i Q_j^T \in SO(3)$.
We give further conditions to have $C_o = C_s$, study error sensitivity highlighting poor robustness when $y_{ij}$ are relative positions in $\R^n$, and illustrate with simulations.


\subsection{Weaker conditions for $C_o = C_s$}\label{ssec:nominal}

The conditions of Theorem~\ref{thm:LgSOn} give too little insight into respective roles of graph and \refy vectors to be efficiently used for network design. The following conditions, weaker than Theorem~\ref{thm:LgSOn}(a), ensure $C_o = C_s$ on $SO(3)$ for small $k$.
These small networks can be used as building blocks for larger networks in which $C_o = C_s$ with generic \refy vectors.\vspace{2mm}

\begin{theorem}\label{prop:SO3:1}
Consider $k$ agents applying (\ref{alg:all}) on $SO(3)$ with fixed $y_{ij}$.
\newline (a) If $k=3$, $\E=\{(1,2),(1,3),(2,3)\}$ and $\;\rk(y_{12}, y_{13}, y_{23}) = \rk(y_{12}, y_{13}) = \linebreak \rk(y_{12}, y_{23}) = \rk(y_{13}, y_{23}) = 2$, then $C_o = C_s$.
\newline (b) If $k=3$, $\E=\{(1,2),(1,3),(2,3)\}$ and $\;\rk(y_{12}, y_{13}, y_{23}) = 3$, then $C_o \not= C_s$. However, output synchronization still implies state synchronization \emph{locally}.
\newline (c) If $k=4$, $\E \supseteq \{(1,2),\,(1,3),\,(1,4),\,(2,3),\,(2,4)\}$,  $\rk(y_a, y_b, y_c) = 3$ for each triple of different $a,b,c \in \E$, and $\frac{(y_{12} \times y_{23})^T(y_{12} \times y_{13})}{y_{23}^T(y_{12} \times y_{13})} \neq \frac{(y_{12} \times y_{24})^T(y_{12} \times y_{14})}{y_{24}^T(y_{12} \times y_{14})}$, then $C_o = C_s$.
\end{theorem}
\begin{proof}
Define relative states $R_j = Q_j Q_1^T$, $j = 2,\ldots,k$. 
State synchronization means $R_j = I_3$, $j = 2,\ldots,k$. 
For (a) and (b), output synchronization means $R_2 y_{12} = y_{12}$, $R_3 y_{13} = y_{13}$ and $R_3 R_2^T y_{23} = y_{23}$. 
For (c), output synchronization requires these on 
subnetwork $A$ of agents $\{1,2,3\}$ and the same conditions on subnetwork $B$ of agents $\{1,2,4\}$.

(a) Output synchronization implies $y_{12}^T R_3^T y_{23} = y_{12}^T R_3^T R_3 R_2^T y_{23} = y_{12}^T y_{23}$. 
The rank conditions yield $y_{23} = a y_{12} + b y_{13}$, with $a, b \in \R\setminus\{0\}$. Thus $y_{12}^T R_3^T y_{23} - y_{12}^T y_{23} = a (y_{12}^T R_3 y_{12} -1) = 0$ which implies $R_3 y_{12} = y_{12}$.
Then $R_3$ has two linearly independent eigenvectors for the eigenvalue $1$, hence $R_3 = I_3$ (see proof of Theorem~\ref{thm:inj}). 
A similar argument yields $R_2 = I_3$.

(b) For any $x \in S^2$ we have $\{R\, x : R \in \st(y_{12})\} = \{y \in S^2 : y_{12}^T\, y = y_{12}^T\, x\}$, and for $x \neq \pm y_{12}$ there is a bijection between $R$ and $y$. 
The rank conditions allow to write $y_{23} = a_1 \e_1 + a_2 \e_2 + a_3 \e_3$ in orthonormal frame $(\e_1,\e_2,\e_3):=(y_{12},\tfrac{y_{12}\times y_{13}}{\Vert y_{12}\times y_{13}\Vert_2}, \tfrac{y_{12} \times (y_{12} \times y_{13})}{\Vert y_{12} \times y_{13}\Vert_2})$, with $a_2 \neq 0$. Define $y_* = a_1 \e_1 - a_2 \e_2 + a_3 \e_3$.  Since $y_{12}^T y_* = y_{12}^T y_{23} \neq \pm 1$, there is a unique $R_{2*} \in \st(y_{12})\setminus{I_3}$ such that $R_{2*}^T y_{23} = y_*$.
Similarly, using $y_{13}^T (y_{12} \times y_{13}) = 0$, there is a unique $R_{3*} \in \st(y_{13})\setminus{I_3}$ such that $R_{3*} y_{23} = y_*$. Then $R_{3*} R_{2*}^T y_{23} = y_{23}$, so $R_{2*}$ and $R_{3*}$ differ from $I_3$ but satisfy all conditions of output synchronization, i.e.~$C_o\setminus C_s \neq \emptyset$.
Furthermore, states in $C_o \setminus C_s$ all correspond to the same relative state $(R_{2*},R_{3*})$. Indeed, assume $y_{\circ} = R_{2\circ}^T \, y_{23} = R_{3\circ}^T \, y_{23}$ for some $(R_{2\circ},R_{3\circ}) \in \st(y_{12}) \times \st(y_{13})$, characterizing output synchronization. The linearly independent requirements $y_{12}^T y_{\circ} = y_{12}^T y_{23}$ and $y_{13}^T y_{\circ} = y_{13}^T y_{23}$ have two solutions in $S^2$: $y_{23}$ and $y_*$. As each admissible $y_{\circ}$ maps to unique $(R_{2\circ},R_{3\circ})\in \st(y_{12}) \times \st(y_{13})$, the only possible $(R_{2\circ},R_{3\circ})$ are $(I_3,I_3)$ and $(R_{2*},R_{3*})$. Thus $C_o$ consists of two separated sets: $C_s$ and $C_o\setminus C_s = \{ (Q, R_{2*}Q , R_{3*}Q ) : Q\in SO(3)\}$.

(c) From case (b), output synchronization in subnetwork $A=\{1,2,3\}$ requires $(R_2,R_3) \in \{(R_{2*},R_{3*}),(I_3,I_3)\}$, with $I_3 \notin \{R_{2*},R_{3*}\}$. A similar argument yields constraint $(R_2,R_4) \in \{(R_{2\bigstar},R_{4*}),(I_3,I_3)\}$, with $I_3 \notin \{ R_{2\bigstar},R_{4*}\}$, for output synchronization in subnetwork $B=\{1,2,4\}$; hereafter we characterize $(R_{2\bigstar},R_{4*})$ by $y_{\bigstar} \in S^2$. Output synchronization in the whole network combines these conditions. Then consistency for $R_2$ requires either $R_2=I_3$ (implying $R_3=R_4=I_3$, state synchronization) or $R_2=R_{2\bigstar} = R_{2*}$. We show that the latter is impossible.
\newline Let $(a_1,a_2,a_3)$ coordinates of $y_{23}$ as in (b), in particular $a_2\neq 0$. Write $y_{24} = b_1 \e_1 + b_2 \e_4 + b_3 \e_5$ in orthonormal frame $(\e_1,\e_4,\e_5)=(y_{12},\tfrac{y_{12}\times y_{14}}{\Vert y_{12}\times y_{14}\Vert_2}, \tfrac{y_{12} \times (y_{12} \times y_{14})}{\Vert y_{12} \times y_{14} \Vert_2})$, with $b_2\neq 0$ by the same argument as in (b). Define $R \in \st(y_{12})$  with $(\e_1,\e_4,\e_5)=R \,(\e_1,\e_2,\e_3)$.
Note that $R_{2*}^T (y_{23}-a_1 \e_1) = (y_*-a_1 \e_1) = (a_2 \e_2 + a_3 \e_3)$ and $R^T  R_{2\bigstar}^T (y_{24}-b_1 \e_1) = R^T (y_{\bigstar}-b_1 \e_1) = (b_2 \e_2 + b_3 \e_3)$. 
Assume $R_{2 \bigstar} = R_{2*}$.
Then $\; \; R_{2*}^T (y_{23}-a_1 \e_1) \times  R^T  R_{2*}^T (y_{24}-b_1 \e_1) $
\begin{eqnarray*}
\!\!\!&& = (a_2 \e_2 + a_3 \e_3) \times (b_2 \e_2 + b_3 \e_3) = m \e_1 = R_{2*}^T\, m\e_1 = R_{2*}^T\, ( \, \vphantom{\sum}(a_2 \e_2 + a_3 \e_3) \times (b_2 \e_2 + b_3 \e_3) \,) \\ 
\!\!\!&& = (R_{2*}^T(a_2 \e_2 + a_3 \e_3)) \times (R_{2*}^T(b_2 \e_2 + b_3 \e_3)) = (-a_2 \e_2 + a_3 \e_3) \times (-b_2 \e_2 + b_3 \e_3) \; .
\end{eqnarray*}
The first and last expressions yield $\tfrac{a_3}{a_2} = \tfrac{b_3}{b_2}$.
Expressing $a_2$, $a_3$, $b_2$, $b_3$ by scalar products leads to $\frac{(y_{12} \times y_{23})^T(y_{12} \times y_{13})}{y_{23}^T(y_{12} \times y_{13})} = \frac{(y_{12} \times y_{24})^T(y_{12} \times y_{14})}{y_{24}^T(y_{12} \times y_{14})}$, violating the last assumption of the statement.
Thus $R_{2 \bigstar} = R_{2*}$ is impossible and $C_o \setminus C_s = \emptyset$.\end{proof}\vspace{2mm}

\begin{remark}\label{rem:SO3thm}
(i) Rank conditions in Theorem~\ref{prop:SO3:1}(b) and (c) are only violated on a set of measure zero for freely chosen $y_{ij} \in S^2$. 
Situation (a) --- coplanar but non-aligned $y_{12}, y_{13}, y_{23}$ --- is only violated on a set of measure zero for $y_{ij} = \tfrac{y_i-y_j}{\pm \Vert y_i-y_j \Vert_2}$, e.g.~\refy vectors which are ``relative positions'', with freely chosen ``positions'' $y_i \in R^3$. 
In Section~\ref{ssec:robustness} we show that state synchronization with (a) is ill-conditioned. Simulations tend to indicate that this is characteristic of \refy vectors defined as ``relative positions''.\newline
(ii) If $k=4$ agents are not connected as in Theorem~\ref{prop:SO3:1}(c), then generically there is a continuum of relative states corresponding to $C_o \setminus C_s$. Indeed: either one agent has a single link and therefore the necessary condition of Theorem~\ref{thm:inj} is not satisfied, 
or agents are connected in square like on Fig.~\ref{fig:RnEx}. In the latter case, output synchronization means $R_1= Q_1 Q_2^T \in \st(y_{12})$, $R_2= Q_2 Q_3^T \in \st(y_{23})$, $R_3= Q_1 Q_4^T \in \st(y_{14})$, $R_4= Q_4 Q_3^T \in \st(y_{34})$ with consistency relation $R_1 R_2 = R_3 R_4$. For generic $y_{ij}$, the first set of constraints allows $R_1 R_2$ and $R_3 R_4$ to span two $2$-dimensional manifolds in $3$-dimensional $SO(3)$; their intersection is generically $1$-dimensional. Thus relative states $(R_1,R_2,R_3)$ characterizing $C_o$ span a one-dimensional manifold (whereas $C_s$ corresponds to $R_1=R_2=R_3=I_3$). 
\end{remark}

\begin{example}\label{Ex:AllComps} The general theorems of Section~\ref{ssec:fix} relate to Theorem~\ref{prop:SO3:1} as follows.\newline
$\bullet$ For the setting of Theorem~\ref{prop:SO3:1}(c),
$\rk L^g \leq \#\E = 6 < n (k-1) = 9$, so the condition of Theorem~\ref{thm:LgSOn}(a) is not satisfied. One checks that the condition of 
Theorem~\ref{thm:LgSOn}(b) holds.\newline
$\bullet$ For the setting of Theorem~\ref{prop:SO3:1}(b), 
condition of Theorem~\ref{thm:LgSOn}(a) is a fortiori still not satisfied, but the condition of Theorem~\ref{thm:LgSOn}(b) holds and ensures that $C_s$ is locally asymptotically stable, despite $C_o \neq C_s$.
In fact $C_o$ has two separated sets, one of which is $C_s$.\newline
$\bullet$ Take, as an example setting of Theorem~\ref{prop:SO3:1}(a), the situation of Example~\ref{ex:nocon} but with \refy vectors $y_{12} = y_{13} = (1,0,0)^T$ and $y_{23} = (0,1,0)^T$.
From Theorem~\ref{prop:SO3:1}(a) and Proposition~\ref{prop:ocstable}, $C_s = C_o$ is locally asymptotically stable. 
However, the condition of Theorem~\ref{thm:LgSOn}(b) is NOT satisfied, so convergence cannot be exponential. \newline
The condition of Theorem~\ref{thm:LgSOn}(a) appears to be too strong for small networks.
\end{example}

Theorem~\ref{prop:SO3:1}(c) gives a generic condition for $C_o = C_s$ on a small graph.
The following algorithm exploits this condition by collapsing repeatedly subnetworks having the graph structure as in~\ref{prop:SO3:1}(c)
to single nodes. If the resulting graphs consists of a single point, then generic \refy vectors yield state synchronization.

\vspace{1ex}
\begin{algo}
\begin{enumerate}\tightlist
\item[$\bullet$] As long as there are
$i,l,m,p \in \V$ with 
$(i,l), (i,m), (p,l), (p,m), (m,l) \in \E$
\begin{enumerate}\tightlist
\item Collapse $i,l,m,p$ to a single vertex $i$ --- the resulting graph can \linebreak have several edges connecting two vertices;
\item As long as there are vertices $i,j$ connected by more than one edge:
\begin{enumerate}\tightlist
\item[] Collapse $i,j$ into a single vertex $i$;
\end{enumerate}
\end{enumerate}
\item[$\bullet$] Output the resulting graph;
\end{enumerate}
\end{algo}

\begin{corollary}\label{cor:algo}
If the graph analysis algorithm above reduces $\G(\V,\E)$ to a single vertex, then $C_o = C_s$ for a generic set of \refy vectors $y_{ij}$.
\end{corollary}
\begin{proof}
If agents $i,s$ are constrained to $Q_i=Q_s$ for a network state in $C_o$, then all their links with the rest of the network are constraints on the same $Q \in SO(3)$. Thus we can collapse $i,s$ into one agent and attribute it all edges incident from $i$ or $s$. 
In addition, if agent $j$ was connected to both $i$ and $s$, then 
there is now a double-constraint $R \, y_{ij} = y_{ij}$ and $R \, y_{is} = y_{is}$ on $R = Q_j Q_i^T = Q_s Q_i^T \in SO(3)$; this generically requires $R = I_3$, so $j$ also belongs to the agents with orientations identical to $Q_i$ at $C_o$.
Our algorithm consecutively collapses agents in this way, according to (i) Theorem~\ref{prop:SO3:1}(c) on $4$-agent groups, assuming generic $y_{ij}$, and (ii) the double-constraint implication just recalled. If this procedure can be applied until one agent remains, then $C_o = C_s$.
\end{proof}

\begin{figure}[t]
	\begin{center}\setlength{\unitlength}{0.7mm}
		\begin{picture}(110,24)
			\multiput(0,12)(20,0){3}{\circle*{3}}  \multiput(70,12)(20,0){3}{\circle*{3}}
			\multiput(10,2)(20,0){2}{\circle*{3}} \multiput(10,22)(20,0){2}{\circle*{3}} \multiput(80,2)(20,0){2}{\circle*{3}} \multiput(80,22)(20,0){2}{\circle*{3}}
			\multiput(10,2)(20,0){2}{\line(1,1){10}} \multiput(80,2)(20,0){2}{\line(1,1){10}}
			\multiput(0,12)(20,0){2}{\line(1,1){10}} \multiput(70,12)(20,0){2}{\line(1,1){10}}
			\multiput(0,12)(20,0){2}{\line(1,-1){10}} \multiput(70,12)(20,0){2}{\line(1,-1){10}}
			\multiput(10,22)(20,0){2}{\line(1,-1){10}} \multiput(80,22)(20,0){2}{\line(1,-1){10}}
			\multiput(0,12)(70,0){2}{\line(1,0){20}} \put(30,2){\line(0,1){20}}
			\put(-5,0){A.} \put(65,0){B.}
		\end{picture}
	\end{center}
	\caption{Non-rigid graphs for which (A.) $C_o = C_s$ from Corollary~\ref{cor:algo} and (B.) $C_o \neq C_s$.}\label{fig:graphs}
\end{figure}
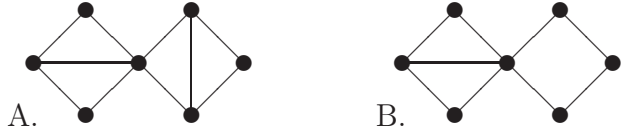

\begin{remark} The condition expressed by our graph analysis algorithm differs from planar formation rigidity, see e.g.~\cite{HendrickxThesis} for a recent study. With $3$ agents, rigidity is a weaker condition as it holds for graphs satisfying the condition of Theorem~\ref{prop:SO3:1}(b). 
With $4$ agents, a graph satisfying Theorem~\ref{prop:SO3:1}(c) is minimal both for (local) rigidity and for $C_o = C_s$. With more agents, our objective becomes easier to achieve: e.g.~the graph on Fig.\ref{fig:graphs}.A ensures $C_o = C_s$ but is clearly not rigid.\newline 
\end{remark}


\subsection{Error sensitivity}\label{ssec:robustness}

Achieving $C_s = C_o$ with less agents for coplanar $y_{ij}$ (Theorem~\ref{prop:SO3:1}(a)) than for linearly independent $y_{ij}$ (Theorem~\ref{prop:SO3:1}(b)) may seem surprising. 
In fact, the two solutions $(Q_2 Q_1^T, Q_3 Q_1^T) \in \{ (I_3,I_3),\, (R_{2*},R_{3*}) \}$ for output synchronization in the proof of Theorem~\ref{prop:SO3:1}(b), correspond to the intersection of circles $\{R_2 y_{23} : R_2 \in \st(y_{12})\}$ and $\{R_3 y_{23} : R_3 \in \st(y_{13})\}$
on $S^2$. 
In the coplanar case, the circles are tangent: this leaves $(I_3,I_3)$ as only solution, but also indicates a possible loss of hyperbolicity
of the dynamical system, which can have undesirable consequences. 
The following formalizes this on basis of error sensitivity. 
In Section~\ref{ssec:simus}, simulations further illustrate the effect on the convergence speed.\vspace{2mm}

We consider measurement errors replacing $Q_i^T y_{ij}$ by
\begin{equation}\label{err-out}
Q_i^T E_{ij} y_{ij}
\end{equation}
with fixed perturbations $E_{ij} \in \{R \in SO(3) : \Vert R-I_3 \Vert_F < \eps \ll 1 \}$. 
The actual closed-loop evolution with these perturbed measurements,
\begin{equation}\label{eq:NoiseDyns}
\tfrac{d}{dt} Q_i = Q_i \sum_{\{j : (i,j) \in \E\}} \, \sk(Q_i^T\, E_{ij} M_{ij} E_{ji}^T\, Q_j) \quad , \quad i=1,2,\ldots,k \; ,
\end{equation}
is the gradient of $f_{o-e} = \sum_{(i,j) \in \E} \, (1-\tr(Q_i^T E_{ij} M_{ij} E_{ji}^T Q_j))$. 
A priori, $E_{ij} \neq E_{ji}$ breaks the output map symmetry and $Q_i = Q_j$ does not yield synchronization of the perturbed outputs. 
For clarity, we denote $C_{o-e}(E_{12},\ldots,E_{k\;(k-1)})$ the output consensus set with perturbed outputs (\ref{err-out}). 
For a large number of
edges, $f_{o-e}$ is strictly positive over $(SO(n))^k$ for generic \refy vectors and $E_{ij}$,
meaning that $C_{o-e} = \emptyset$.
We define 
$$\dist^2(C_A , C_B ) := \min_{\substack{(Q_{1A},Q_{2A},\ldots,Q_{kA}) \in C_A,\\ (Q_{1B},Q_{2B},\ldots,Q_{kB}) \in C_B}} \; \sum_{i=1}^k \Vert Q_{iA} - Q_{iB} \Vert_F^2 \, .$$

\begin{theorem}\label{prop:rob}
Consider $k=3$ agents applying (\ref{eq:NoiseDyns}) on $SO(3)$. 
Denote by $E_\eps$ the set $\{(E_{12},\ldots,E_{32}) : \|I_3-E_{ij}\|_F < \eps \; \forall i,j \}\subset
(SO(3))^{k(k-1)}$.
\newline (a) In setting of Theorem~\ref{prop:SO3:1}(a), there is a constant $c>0$ such that for any sufficiently small $\eps>0$ there is a non-zero measure set
$S_\eps\subset E_\eps$ where 
$C_{o-e}(E_{12},\ldots,E_{32})$ is nonempty and
$\; \dist(C_s,C_{o-e}(E_{12},\ldots,E_{32})) \ge c \sqrt{\eps} \;$ for all $(E_{12},\ldots,E_{32})\in S_\eps$. 
\newline
Furthermore, for any $\eps>0$ there is also a non-zero measure set $N_\eps \subset E_\eps$ such that $C_{o-e}(E_{12},\ldots,E_{32}) =\emptyset$ for all $(E_{12},\ldots,E_{32})\in N_\eps$.
\newline (b) In setting of Theorem~\ref{prop:SO3:1}(b), there is a constant $c > 0$ such that, for any sufficiently small $\eps>0$, $C_{o-e}(E_{12},\ldots,E_{32})$ is non-empty and $\; \dist(C_s, C_{o-e}(E_{12},\ldots,E_{32}))\le c\, \eps \;$ for all $(E_{12},\ldots,E_{32})\in E_\eps$.
\end{theorem}
\begin{proof}
	The simple but tedious proof, constructing specific sets, is in the appendix.
\end{proof}\vspace{2mm}

So in the setting of Theorem~\ref{prop:SO3:1}(a), the ratio between state synchronization error\linebreak $\dist(C_s,C_{o-e})$ and measurement error $\eps$ can grow infinite for $\eps \rightarrow 0$, while in the setting of Theorem~\ref{prop:SO3:1}(b) this ratio remains finite; in this sense only the second situation is robust. From geometric considerations (i.e.~set intersections, see proof of Theorem~\ref{prop:rob}), we also expect good robustness in the setting of Theorem~\ref{prop:SO3:1}(c).
Simulations suggest that this distinction --- $y_{ij}$ unrestricted robust, $y_{ij}$ restricted to relative positions not robust --- 
carries over to Corollary~\ref{cor:algo}.


\subsection{Numerical Simulations}\label{ssec:simus}

In this section we present some numerical simulations for $SO(3)$.

Figure~\ref{fig:2a} shows local evolution of $f_o$ and $f_s$ for $6$ all-to-all connected agents implementing~(\ref{alg:all}). The thick lines are for generic $y_{ij}$ (denoted B, analogous to case(b) of Theorems~\ref{prop:SO3:1} and~\ref{prop:rob}), thin lines are for $y_{ij}$ derived from generic positions $y_i \in \R^3$, that is $y_{ij} = \tfrac{y_i-y_j}{\pm \Vert y_i-y_j \Vert_2}$ (denoted A, analogous to case(a) of the theorems). 
Top plots are without measurement errors: for B, $f_o$ and $f_s$ exponentially decrease to machine precision,  for A the decay is clearly non-exponential. Note that the conditions of Theorem~\ref{thm:LgSOn} both fail for A, while they hold for B. 
Bottom plots repeat the simulation with random measurement errors $\Vert E_{ij} - I_3 \Vert_F \leq \eps= 0.01$.
As $C_{o-e}$ cannot be reached exactly, $f_{o-e}$ stabilizes between $10^{-4}$ and $10^{-3}$ for both cases.
Case A now actually \emph{saturates} at a significantly higher value of $f_s$ than B. This is reminiscent of Theorem~\ref{prop:rob}: for case(b) the synchronization error is of order $\eps$, for case(a) it can be of order $\sqrt{\eps} > \eps$. 
The difference between subexponential convergence and saturation for A, becomes apparent on exceedingly long simulations (displayed on the small plots). The setting with $y_{ij}$ given by relative positions seems ill-behaved in practice.

\begin{figure}[t]
\setlength{\unitlength}{1mm}
\begin{picture}(180,99)
\put(-10,-2){\includegraphics[width=145mm]{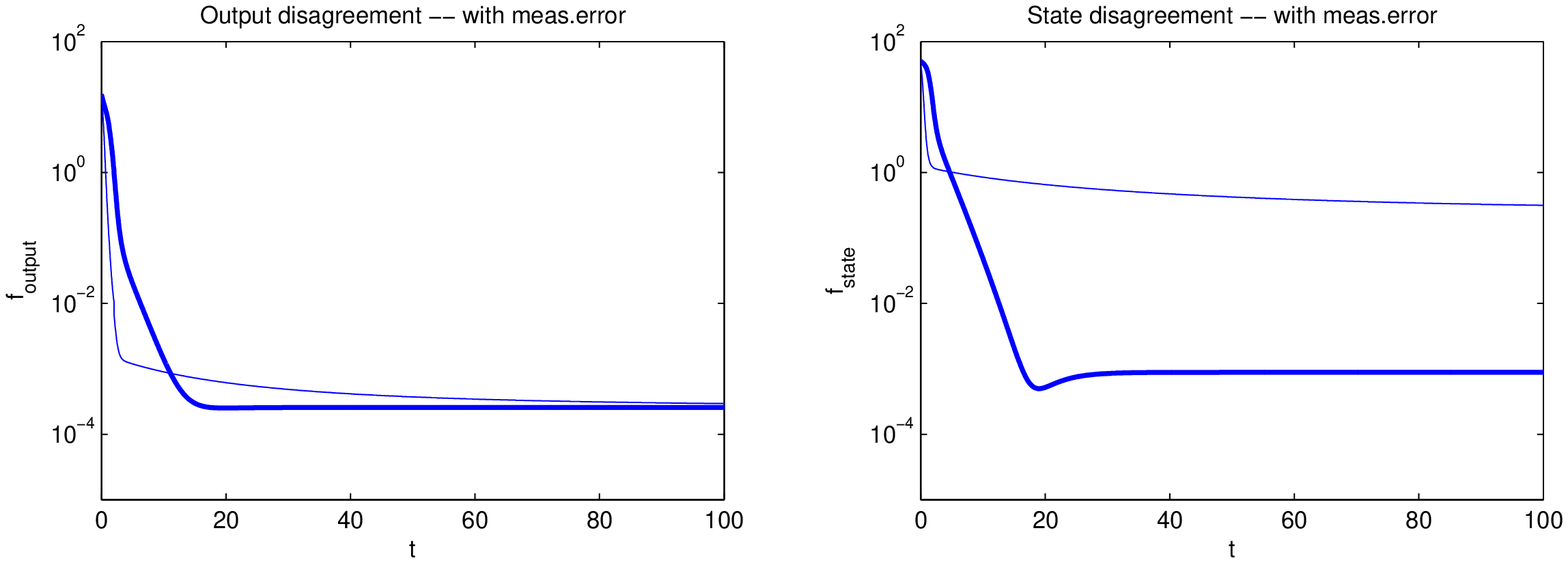}}
\put(-10,49){\includegraphics[width=145mm]{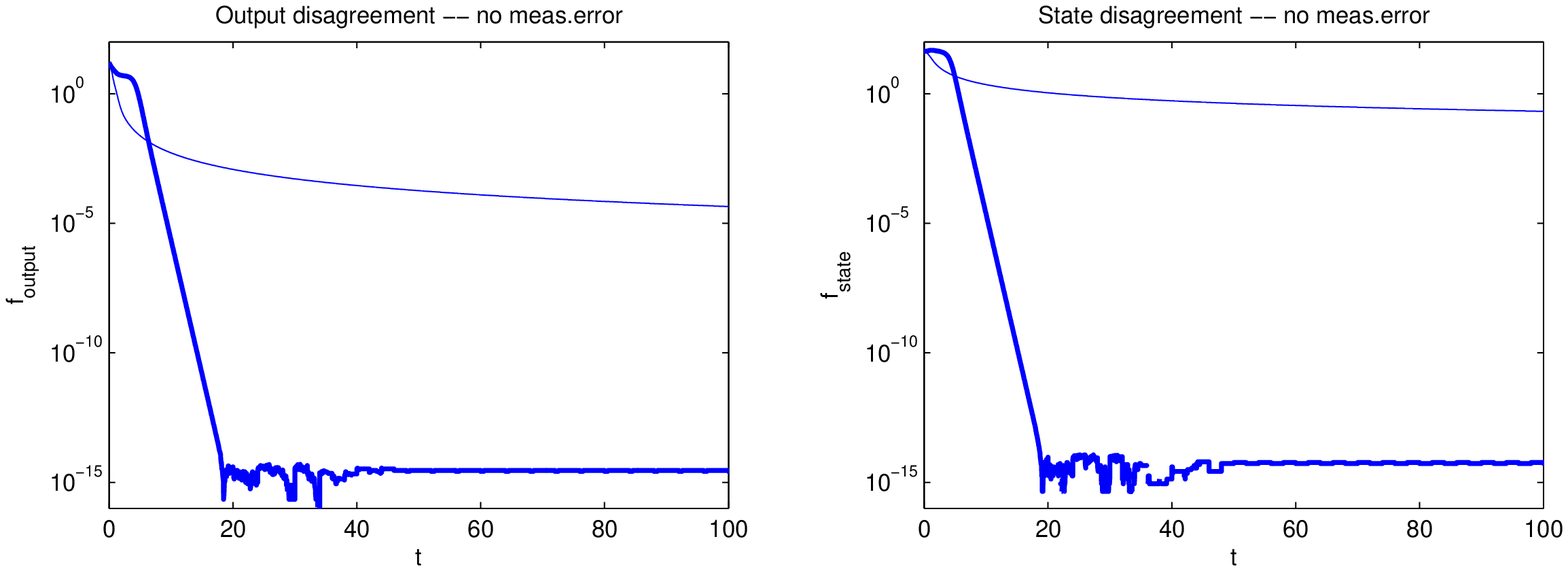}}
\put(122,20){\includegraphics[width=25mm]{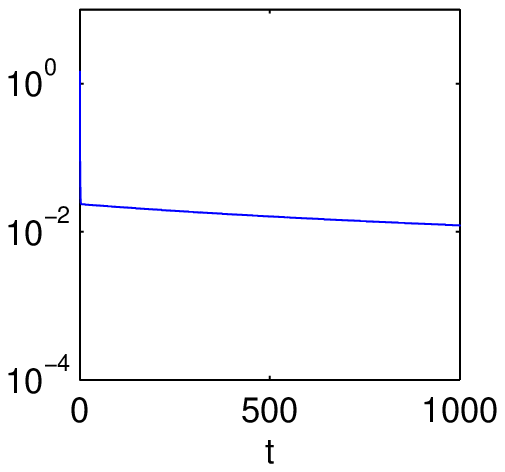}}
\put(122,71){\includegraphics[width=25mm]{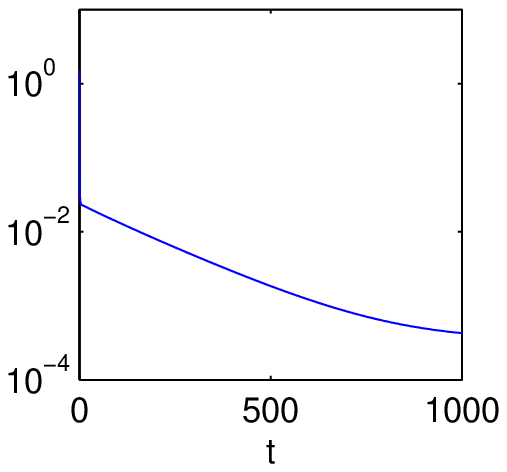}}
\end{picture}
\caption{Evolution of $f_{o-e}$ and $f_{s}$ for $6$ all-to-all connected agents, random initial conditions: thick lines generic $y_{ij}$ (like case(b) of Th.\ref{prop:SO3:1}), thin lines $y_{ij} = \tfrac{y_i-y_j}{\pm \Vert y_i-y_j \Vert_2}$ with generic $y_i$ (like case(a) of Th.\ref{prop:SO3:1}). Top plots without measurement errors, bottom plots with measurement errors. Far right plots: pursuing the simulation of case(a) to larger times.}\label{fig:2a}
\end{figure}

\begin{figure}[h!]
	\setlength{\unitlength}{1mm}
	\begin{picture}(160,45)
	\put(0,-3){\includegraphics[width=145mm]{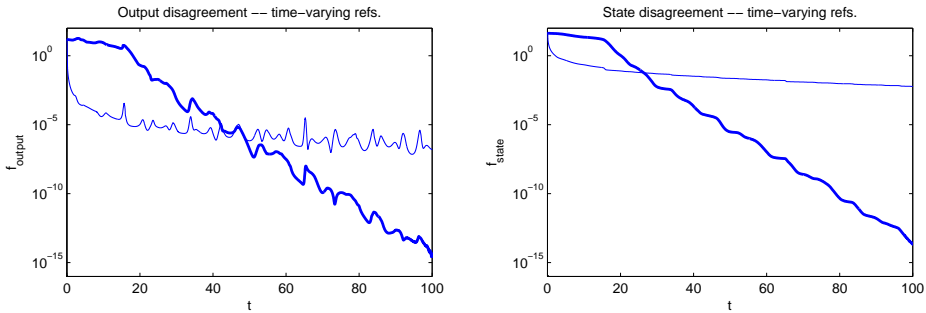}}
\end{picture}
	\caption{Evolution of $f_{o}$ and $f_{s}$ for $6$ all-to-all connected agents, random initial conditions and $y_{ij} = \tfrac{y_i-y_j}{\pm \Vert y_i-y_j \Vert_2}$ with time-varying $y_i(t)$. No measurement errors. Thick lines $\eps=0.1$ in (\ref{eq:tv}); thin lines $\eps=10$.}\label{fig:tv}
\end{figure}

Figure~\ref{fig:tv} illustrates behavior with time-varying $y_{ij}$. Again we consider $6$ all-to-all coupled agents (no measurement errors), with \refy vectors derived from generic positions in $\R^3$, that is $y_{ij} = \tfrac{y_i-y_j}{\pm \Vert y_i-y_j \Vert_2}$. The $y_i(t)$ vary quasi-periodically in time with frequencies in $[0.05,\tfrac{1}{2\pi}]$~Hz. For $\eps=0.1$ in~(\ref{eq:tv}), that is the thick lines, roughly exponential convergence to $f_o=f_s=0$ is observed, with oscillations from time-variation. Note that this $\eps$ is not exceedingly small w.r.t.~frequencies of $y_i(t)$. The value $\eps=10$ (thin line) is too large for exploiting the time-varying setting: corresponding curves look more like the ones for fixed $y_{ij}$ (thin lines on Figure~\ref{fig:2a}), converging to $0$ very poorly, if at all.

Finally, we illustrate the influence of graph $\G$ on convergence. Figure~\ref{fig:fin} shows local evolution of $f_o$ and $f_s$ for $7$ agents interconnected according to graphs A and B of Fig.~\ref{fig:graphs}, page~\pageref{fig:graphs}. Both runs use the same generic fixed $y_{ij}$ (except for the missing link in graph B) and same initial conditions in the neighborhood of $C_s$. The systems similarly converge to $f_o=0$ for both graphs. For graph A (thick lines), in accordance with Corollary~\ref{cor:algo}, $f_s$ also converges to $0$. For graph B (thin lines), $f_s$ converges to an arbitrary value, as $C_o \setminus C_s$ admits a continuum of relative states $(Q_2 Q_1^T,\ldots,Q_7 Q_1^T) \neq (I_3,\ldots,I_3)$.

\begin{figure}[ht]
		\setlength{\unitlength}{1mm}
		\begin{picture}(160,45)
		\put(0,-3){\includegraphics[width=145mm]{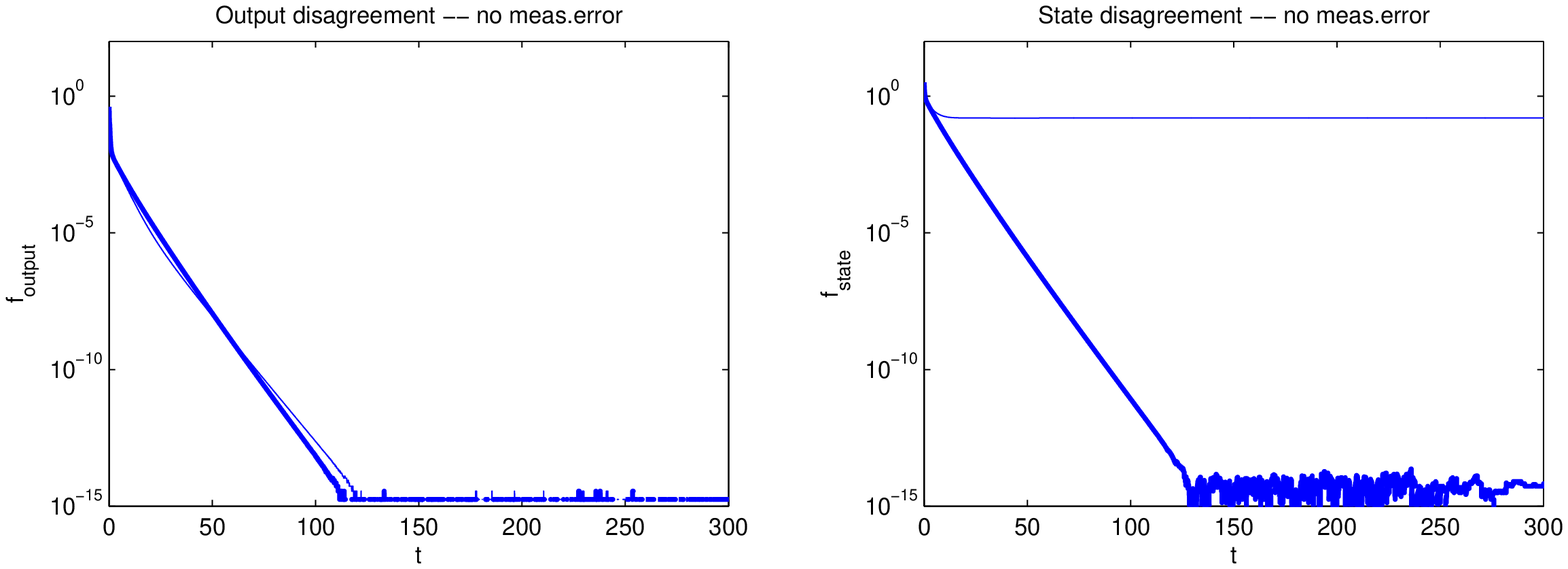}}
	\end{picture}
\caption{Evolution of $f_{o}$ and $f_{s}$ for $7$ agents with generic $y_{ij} \in S^2$ and initial conditions in the neighborhood of $C_s$, with interconnection graph of Fig.~\ref{fig:graphs}.A (thick lines) and of Fig.~\ref{fig:graphs}.B (thin lines). No measurement errors.}\label{fig:fin}
\end{figure}


\section{Conclusion}\label{sec:conc}

The present paper studies autonomous synchronization of orientation matrices $Q_i \in SO(n)$ with coupling through outputs $\neq$ states: information exchanged on each link is the expression of a vector in respective body frames. A natural gradient law generalizing full-state coupling case~\cite{ss:08} is proposed. Its convergence properties feature interesting specificities, even locally. These are studied in detail, with both necessary conditions and sufficient ones for state synchronization. An algorithm is proposed to find/build networks in which output synchronization implies state synchronization on $SO(3)$ for generic \refy vectors. The setting with \refy vectors derived as relative positions in $\R^3$ is shown to be non-robust. An analog system with states $x_i \in \R^n$ is studied for comparison.

We believe that this study can be both of practical interest --- to design synchronizing networks with less than full-state exchange --- and of theoretical interest --- illustrating non-robust cases, highlighting further differences between $\R^n$ and $SO(n)$, and maybe motivating a more detailed study of orbit intersections on homogeneous spaces and Lie groups. 
Regarding distributed systems, the non-trivial issues raised in this work motivate further study of observability-related theories for agents interacting in a network.


\section{Appendix}


\begin{lemma} \label{lem:BasicConv}
Consider an equilibrium $\, \bar{Q}:=(\bar{Q}_1,\ldots,\bar{Q}_k) \,$ of (\ref{alg:all}). 
Define 
$C_s^r := \{(X, \ldots, X) \in (\R^{n\times n})^k \, : X \in \R^{n \times n} \}$ and $F_{ij} = \bar{Q}_i^T M_{ij} \bar{Q}_j$. 
Build the symmetric matrix $L^{gg} = (L^{gg}_{ij}) \in\R^{kn\times kn}$ with
$L^{gg}_{ij} = - a_{ij} F_{ij} \text{ for } i \neq j \; \text{ and } \; L^{gg}_{ii} = {\textstyle \sum_j} \; a_{ij} \, F_{ij} .$
Then $S = \{ (\bar{Q}_1 Q^*,\ldots, \bar{Q}_k Q^*) : Q^* \in SO(n) \}$ is a set of equilibria, and $S$ is locally exponentially stable if and only if $(I_n \otimes L^{gg})$ is positive definite on the subspace of $\vect(\so (n)^k)$ orthogonal to $\vect(C_s^r)$. \newline 
(N.B.: the condition is \emph{necessary} only for \emph{exponential} stability.) 
\end{lemma}
\begin{proof}
At a critical point $x$ the Hessian quadratic form $H_f(x)\colon (T_x \M \times T_x \M) \rightarrow \R$ 
of a smooth function $f\colon \M\rightarrow\R$ on a smooth manifold $\M$ is fully defined by $H_f(x) (v,v):= \frac{d^2}{dt^2}(f\circ \gamma)(0)$ for any smooth curve $\gamma(t)$ on $\M$ with $\gamma(0)=x$ and $\tfrac{d}{dt}\gamma\vert_0=v$, see \cite{hm}.
On a Riemannian manifold one can define at critical points\footnote{For other points one would have to use the Riemannian connection.} the self-adjoint Hessian operator $\mathtt{H}_f(x)\colon T_x\M\rightarrow T_x\M$ by 
$\langle v, \mathtt{H}_f(x) v\rangle = H_f(x) (v,v)$ for $v\in T_x\M$ with $\langle\cdot,\cdot\rangle$ the Riemannian metric. 

Take $\gamma\colon\R\rightarrow (SO(n))^k$ with $\gamma(0) = (\bar{Q}_1,\ldots,\bar{Q}_k)$ and $\frac{d}{dt}\gamma(0) = (\bar{Q}_1 \Omega_1, \ldots , \bar{Q}_k \Omega_k) =: \bar{Q} \bar{\Omega} \, \in T_{\bar{Q}}(SO(n))^k$, with $\Omega_i \in \so(n)$; also denote $\bar{\Omega}=(\Omega_1, \ldots, \Omega_k) \in \R^{n \times nk}$. 
From (\ref{PartialCost}),
$$H_{f_o}(\bar{Q}) (\bar Q\bar{\Omega},\bar Q\bar{\Omega}) = \tr(\bar{\Omega} L^{gg} \bar{\Omega}^T) = (\vect(\bar{\Omega}^T))^T (I_n \otimes L^{gg}) (\vect(\bar{\Omega}^T)) \; .$$
Right-invariance of (\ref{alg:all}) implies that $S$ is a set of equilibria on which $f_o = f_o(\bar{Q})$ is constant.
Moreover, curves with tangent vector in $\widetilde{Q} (\so (n)^k \cap C_s^r) = \{(Q_1\Omega_b,\ldots, Q_k\Omega_b) : \Omega_b \in \so(n) \}$ at any point $\widetilde{Q} = (Q_1,\ldots,Q_k) \in (SO(n))^k$ follow invariance directions of the dynamics. 
In particular, $S$ is a submanifold generated by all of these curves that pass through $\bar{Q}$.
The condition on $(I_n \otimes L^{gg})$ ensures that $H_{f_o}(\bar{Q}) (\bar Q\bar{\Omega},\bar Q\bar{\Omega}) = 0$ implies $\bar{Q}\bar{\Omega} \in \bar{Q}(\so (n)^k \cap C_s^r) = T_{\bar{Q}} S$. 

As in the proof of Theorem~\ref{thm:tv} we project the dynamics onto the reductive homogeneous space $\M = (SO(n))^k/C_s$ and
denote by $\pi\colon (SO(n))^k\rightarrow \M$ the canonical projection. 
The function $f_o$ is constant on the fibers of $\pi$ and induces a smooth function $\hat f_o$ on $\M$ which satisfies $f_o = \hat f_o\circ \pi$. Since \eqref{alg:all} is invariant under right multiplication of all $Q_i$ by the same $Q\in SO(n)$, it induces a vector field $X$ on $\M$.
The set $S$ is collapsed by $\pi$ onto a single equilibrium $s\in \M$ and it is locally exponentially stable under (\ref{alg:all}) iff $s$ is locally exponentially
stable under $X$. We equip $\M$ with the induced normal Riemannian metric $\langle \cdot,\cdot\rangle_\M$ and denote by $\mathfrak{h}(\widetilde Q)$, $\widetilde Q\in(SO(n))^k$, 
the canonical horizontal distribution on $(SO(n))^k$, i.e.~the orthogonal complement of the vertical distribution 
$\mathfrak{v}(\widetilde Q) := \ker (T_{\widetilde Q} \pi)= \widetilde{Q} (\so (n)^k \cap C_s^r)$ w.r.t. $\langle \cdot,\cdot\rangle_\M$.
By the invariance of $f_o$ under the action of $C_s$ on $(SO(n))^k$ by right multiplication, 
the vertical part $(\grad f_o(\widetilde Q))^{\mathfrak{v}}\in \mathfrak{v}(\widetilde Q)$ of $\grad f_o(\widetilde Q)$ is zero. 
Recall that for the normal metric, $\langle V,W \rangle_{(SO(n))^k} = \langle T_{\widetilde Q} \pi(V), T_{\widetilde Q} \pi(W) \rangle_\M$ for all $\widetilde Q\in (SO(n))^k$, $V, W \in \mathfrak{h}(\widetilde Q)$~\cite{ce}. 
Some calculations, using properties of the normal metric and the definition of the gradient, show that
the projected vector field $X(\pi(\widetilde Q))$ is actually $\grad_{\M} \hat f_o(\pi(\widetilde Q))$, the gradient of $\hat f_o$ w.r.t.~the normal metric.
So $\grad f_o$ projects to a gradient system on $\M$. 
Therefore, the equilibrium point $s \in \M$ is locally exponentially stable if and only
if the Hessian form of $\hat f_o$ at $s$ is strictly positive definite (this follows directly from standard results on exponential stability and the fact that the Hessian operator $\mathtt{H}_{\hat f_o}$ at a critical point corresponds to linearization of the vector field in exponential charts).
Simply using the definition of the Hessian form $H_{\hat f_o}$ and lifting curves from $\M$ to $SO(n)^k$ horizontally,
we see that $H_{\hat f_o}$ is strictly positive definite at $s$ if and only if $H_{f_o}$ is strictly positive definite on 
$\mathfrak{h}(\bar Q)$.
But from the above, $H_{f_o}$ is strictly positive definite on $\mathfrak{h}(\bar Q)$ iff $(I_n \otimes L^{gg})$ is positive definite on the subspace of $\vect(\so (n)^k)$ orthogonal to $\vect(C_s^r)$.
\end{proof}

The following result is used to prove Theorem~\ref{prop:rob}. Let $h\colon SO(3)\rightarrow S^2$ given by $h(Q)=Q^T y$ for $y \in S^2$ and denote $h^{-1}(p) = \{Q : Q^T y = p\}$ for $p \in S^2$.
\begin{proposition}\label{prop:pertfibercut}
If $y,z \in S^2$ are linearly independent, then there is a neighborhood $\mathcal{U}$ of $I_3$ in $SO(3)$ such that 
for all $p\in S^2$, $R\in \mathcal{U}$, $S\in SO(3)$, the set $ h^{-1}(p) \cap R \st(z) S$ contains at most one element.
\end{proposition}
\begin{proof}
Note that $R \st(z) R^T = \st(R z)$.
One checks that $h^{-1}(p)S^T R^T= h^{-1}(R S p)$.
Thus $\left(h^{-1}(p) \cap R \st(z) S\right) = \left(h^{-1}(RSp) \cap \st(Rz) \right) RS$.
Given linearly independent $z,y$, there is a neighborhood $\mathcal{U}$ of $I_3$ in $SO(3)$ 
such that $y$ and $v := R z$ remain linearly independent for all $R\in \mathcal{U}$.
This linear independence implies that equation $X^T y = RS p$ has at most one solution $X\in\st( v ) = \st( R z )$ for all $p\in S^2$.
\end{proof} \vspace{2mm}

\noindent \textbf{Proof of Theorem~\ref{prop:rob}:} 
$C_{o-e}$ is invariant under right multiplication of the states by a common constant, so (to simplify notations) we can characterize $C_{o-e}$ by first assuming $Q_1=I_3$ and at the end reintroducing a common right-multiplication.
A (hypothetical) state $(I_3,Q_2,Q_3)$ in the perturbed output consensus set $C_{o-e}$ corresponds to 
$Q_2 E_{12} y_{12} = E_{21} y_{12}$, $Q_3 E_{13} y_{13} = E_{31} y_{13}$ 
and $Q_3 Q_2^T E_{23} y_{23} = E_{32} y_{23}$.
The first two constraints yield $Q_2\in E_{21} \st(y_{12}) E_{12}^T$ and $Q_3\in E_{31} \st(y_{13}) E_{13}^T$.
The output of agent $2$ for $y_{23}$
is then contained in the circle
$$
Z'_2 = \{
E_{12} Q  E_{21}^T E_{23} y_{23} : Q \in  \st(y_{12})  \}
= \{ y\in S^2 : y_{12}^T E_{12}^T 
y = y_{12}^T E_{21}^T E_{23} y_{23} \}.
$$
Similarly, output of agent $3$ for $y_{23}$ 
is contained in the circle
$$
Z'_3 = \{
E_{13} Q  E_{31}^T E_{32} y_{23} : Q \in  \st(y_{13})  \}
= \{ y\in S^2 : y_{13}^T E_{13}^T   
y = y_{13}^T E_{31}^T E_{32} y_{23} \}.
$$
For a state in $C_{o-e}$, these two outputs must be equal, so they must belong to $Z'_2 \cap Z'_3$. Since the pairs $y_{12},y_{23}$ and $y_{13}, y_{23}$ are each linearly independent, each point $p \in Z'_2 \cap Z'_3$ corresponds to exactly one state of the form $(I_3,Q_2,Q_3) \in C_{o-e}$ for sufficiently small perturbations $\Vert E_{ij} - I_3 \Vert_F$, by Proposition~\ref{prop:pertfibercut} (existence of the state is ensured by definition of $Z'_2$ and $Z'_3$). Moreover, $y_{12},y_{13}$ linearly independent ensures that $Z'_2 \cap Z'_3$ can contain at most two points for sufficiently small $\Vert E_{ij} - I_3 \Vert_F$.
Let us introduce $G_3 := E_{31}^T Q_3 E_{13}$, $G_2 := E_{13}^T E_{12} E_{21}^T Q_2 E_{13}$, $\bar{y}_{12} = E_{13}^T E_{12} y_{12}$, $\bar{y}_{23} = E_{31}^T E_{32} y_{23}$ and $E_A = E_{13}^T E_{12} E_{21}^T E_{23} E_{32}^T E_{31}$. A state being in $C_{o-e}$ then requires $(G_{2},G_{3}) \; \in \; \st(\bar{y}_{12}) \times \st(y_{13})$.
 Further, define
\begin{align*}
Z_2 &:= E_{13}^T Z_2' = \{y \in S^2 : \bar{y}_{12}^T y = \bar{y}_{12}^T (E_A \bar{y}_{23})\}\; , \quad
Z_3 &:= E_{13}^T Z_3' = \{y \in S^2 : y_{13}^T y = y_{13}^T \bar{y_{23}} \} \; .
\end{align*}
Again, there is a bijection between points in $Z_2\cap Z_3$ and states in $C_{o-e}$ with $Q_1=I_3$.
We also define the balls
\begin{align*}
B_2 := \{y \in S^2 : \bar{y}_{12}^T y \geq \bar{y}_{12}^T (E_A \bar{y}_{23})\}\; , \quad
B_3 := \{y \in S^2 : y_{13}^T y \geq y_{13}^T \bar{y}_{23} \}
\end{align*}
whose boundaries are $Z_2$ resp.~$Z_3$. Furthermore, we denote by $Z_{2o}, Z_{3o}$ the circles and $B_{2o}, B_{3o}$ the balls for the unperturbed case, i.e.~for $E_{ij}=I_3$. We restrict ourselves to the case $0 < y_{12}^T y_{13} < y_{12}^T y_{23} \leq y_{13}^T y_{23}$; the proof for other cases is strictly analogous.
In addition, we redefine $\e_1 := \tfrac{\bar{y}_{12} \times y_{13}}{\Vert \bar{y}_{12} \times y_{13} \Vert_2}$ and $\hat{y}_{23} = \tfrac{(I_3 - \e_1 \e_1^T) \bar{y}_{23}}{\Vert (I_3 - \e_1 \e_1^T) \bar{y}_{23} \Vert_2}$.

From invariance of the Frobenius distance, $\Vert I_3 - J_2 \Vert_F \leq \Vert I_3 - J_1 \Vert_F + \Vert I_3 - J_1^T J_2 \Vert_F$ for any $J_1,J_2 \in SO(3)$. This allows to show that for $(E_{12},\ldots,E_{32}) \in E_\eps$,
\begin{equation}\label{eq:niceprop}
\Vert Q - J_E \Vert_F \leq m \eps \text{ if } J_E \text{ is the product of } Q \text{ with } m \text{ error rotations } E_{ij}
\end{equation} 
arbitrarily split into $m_1$ left and $m_2$ right multiplications, $m_1+m_2=m$.
Also, $x_1^T (I_3 - J_1) x_2 \leq \Vert I_3 - J_1 \Vert_F$ for $J_1 \in SO(n)$, $x_1,x_2 \in S^2$.

(a) For Theorem \ref{prop:SO3:1}(a), circles $Z_{2o}$ and $Z_{3o}$ intersect at a single tangency point $y_{23}$; then (in our restricted case, else balls must be redefined) the same is true for balls $B_{2o}$ and $B_{3o}$.\newline
First consider the second assertion of the theorem. For any $\eps>0$
we can find a $\beta>0$ and an open subset $N_\eps\subset E_\eps$ of $(SO(3))^6$ such that 
$y_{13}^T \bar{y}_{23} > y_{13}^T y_{23}$, $\bar{y}_{12}^T E_A \bar{y}_{23} > y_{12}^T y_{23} + \beta \;$ and 
$\bar{y}_{12}^T y_{12} \geq 1- \tfrac{\beta}{2}$ for all perturbations $(E_{12},\ldots,E_{32}) \in N_\eps$. 
Then $Z_2$ and $Z_3$ are in the interior of $B_{2o}$ and $B_{3o}$ respectively, so they cannot intersect and $C_{o-e} = \emptyset$.

Now consider the first assertion.
For any $\eps>0$ sufficiently small, we can find an $\alpha > 0$ and an open subset $N_\eps\subset E_\eps$ of $(SO(3))^6$ such that 
\begin{align*}
&(i) && y_{13}^T \bar{y}_{23} < y_{13}^T E_A \bar{y}_{23} - \alpha \eps , \\
&(ii) && \bar{y}_{12}^T \bar{y}_{23} > \bar{y}_{12}^T E_A \bar{y}_{23} + \alpha \eps \;\;\; \text{ and } \\
&(iii) && \Vert \bar{y}_{23} - \hat{y}_{23} \Vert_2 < \tfrac{\alpha}{2} \eps
\end{align*}
for perturbations $(E_{12},\ldots,E_{32}) \in N_\eps$. \newline
From (i), $E_A \bar{y}_{23}$, on the boundary of $B_{2}$, also belongs to interior of $B_{3}$, so $B_2 \cap B_3$ has non-zero measure and $Z_2 \cap Z_3$ contains two points, corresponding to two separated sets in $C_{o-e}$.
Take a state $(Q_1^*,Q_2^*,Q_3^*) \in C_{o-e}$ with corresponding $(G_{2*},G_{3*})$ and $y_* = G_{3*} \bar{y}_{23} = G_{2*} E_A \bar{y}_{23}$. 
Then (ii) implies $\bar{y}_{12}^T y_* = \bar{y}_{12}^T G_{3*} \bar{y}_{23} = \bar{y}_{12}^T E_A \bar{y}_{23} < \bar{y}_{12}^T \bar{y}_{23} - \alpha \eps$ so $\bar{y}_{12}^T (I_3 - G_{3*}) \bar{y}_{23} > \alpha \eps$ and using (iii) we get $\bar{y}_{12}^T (I_3 - G_{3*}) \hat{y}_{23} > \tfrac{\alpha}{2} \eps$.
Writing $G_{3*}$ and vectors in orthonormal basis $(\e_1,y_{13},\e_1 \times y_{13})$ shows that this requires $G_{3*}$ to be a rotation of 
angle $\theta(\eps)$ with $(1-\cos(\theta(\eps))) > c_1\eps$ for some constant $c_1$ independent of $\eps$; then there is a constant $c_2>0$ independent of $\eps$ such that $\theta(\eps) > c_2 \sqrt{\eps}$.
Thus we can find, for $\eps$ sufficiently small, a constant $c_3 > 0$ independent of $\eps$ such that 
$\Vert I_3 - G_{3*} \Vert_F = \sqrt{8} \sin(\tfrac{\theta(\eps)}{2}) > c_3\sqrt{\eps}$.
Using property (\ref{eq:niceprop}) and
\begin{multline*}
\dist(C_s,C_{o-e}) =
\min_{\substack{Q\in SO(n) \\ (I_3,Q_2,Q_3)\in C_{o-e}}} \left( \|Q-I_3\|^2_F + \|Q-Q_2\|^2_F + \|Q-Q_3\|_F^2\right)^{1/2}\\
\ge \tfrac{1}{\sqrt{3}} \min_{\substack{Q\in SO(n) \\ (I_3,Q_2,Q_3)\in C_{o-e}}} \|Q-I_3\|_F + \|Q-Q_2\|_F + \|Q-Q_3\|_F \\
\ge \tfrac{1}{\sqrt{3}} \min_{\substack{Q\in SO(n) \\ (I_3,Q_2,Q_3)\in C_{o-e}}} \|Q-I_3\|_F + \|Q-Q_3\|_F
\ge \tfrac{1}{\sqrt{3}} \min_{(I_3,Q_2,Q_3)\in C_{o-e}} \|I_3-Q_3\|_F 
\end{multline*}
we get, for sufficiently small $\eps$, a constant $c_4>0$ independent of $\eps$ such that 
$$
\dist(C_s,C_{o-e}) \ge \tfrac{1}{\sqrt{3}}\Vert I_3 - Q_{3*} \Vert_F \geq
\tfrac{\Vert I_3 - G_{3*} \Vert_F - \Vert G_{3*}- Q_{3*} \Vert_F}{\sqrt{3}} \geq
\tfrac{1}{\sqrt{3}}\Vert I_3 - G_{3*} \Vert_F  - 2 \eps \geq c_4 \sqrt{\eps} \, .
$$ 

(b) Denote by $\theta$ the angle between $y_{23}$ and $y_{12}$ i.e.~$y_{12}^T y_{23}=\cos(\theta)$, with $\theta \in (0,\tfrac{\pi}{2}]$ since we consider $y_{12}^T y_{23}>0$.
Define $B_{\alpha} = \{ y \in S^2 : \Vert y - y_{23} \Vert_2 \leq \alpha \}$.
Theorem~\ref{prop:SO3:1}(b) requires $Z_{2o}$ and $Z_{3o}$ to intersect at two points. This requires their intersections to be transversal. As a consequence, there exist $\beta_0,\gamma_0 > 0$ such that $\alpha < \gamma_0$ implies: (i) following the boundary of $B_{\alpha}$, one passes through 4 points $p_{2a},p_{3a},p_{2b},p_{3b}$ in that order, with $\{ p_{2a},p_{2b} \} \subset Z_{2o}$ and $\{ p_{3a},p_{3b} \} \subset Z_{3o}$; and (ii) on each boundary arc $[p_i,p_j]$, there is at least one point $r_{[p_i,p_j]}$ whose distances to $Z_{2o}$ and $Z_{3o}$ are both strictly larger than $\beta_0 \alpha$.
Next, note that there can be no point belonging to $Z_2$ (resp.~$Z_3$) whose distance to $Z_{2o}$ (resp.~to $Z_{3o}$) is larger than $\tfrac{4 \eps}{\sin(\theta/2)}$. Indeed, denote $x \in S^2$ a point at distance $\delta$ from $Z_{2o}$ in $\mathbb{R}^3$, and denote $\phi \in [0,\pi]$ the angle between vectors $x$ and $y_{12}$; then $\vert \cos(\theta) - \cos(\phi) \vert \geq \delta \sin(\tfrac{\theta+\phi}{2}) \geq \delta \sin(\tfrac{\theta}{2})$. At the same time, $x \in Z_2$ requires $\bar{y}_{12}^T (x-E_A \bar{y}_{23})=0$ hence
\begin{eqnarray*}
	\vert \cos(\theta) - \cos(\phi) \vert & = & \vert y_{12}^T (x - y_{23}) \; - \; \bar{y}_{12}^T (x-E_A \bar{y}_{23}) \vert \\
	& = & \vert y_{12}^T (I_3-E_{12}^T E_{13}) x + y_{12}^T (I_3 - E_{21}^T E_{23}) y_{23} \vert \\
	& \leq & \Vert y_{12} \Vert_2 (\Vert I_3 - E_{12}^T E_{13} \Vert_F \Vert x \Vert_2 + \Vert I_3 - E_{21}^T E_{23} \Vert_F \Vert y_{23} \Vert_2) \; \leq \; 4 \eps \, .
\end{eqnarray*}
These two inequalities yield $\delta < 4\eps / \sin(\tfrac{\theta}{2})$. A similar argument applies for $Z_3$.

As $Z_2$ and $Z_3$ are thus confined to neighborhoods of $Z_{2o}$ and $Z_{3o}$, take $\alpha < \gamma_0$ and $\eps = \tfrac{\beta_0 \alpha \sin(\theta/2)}{4}$. Then the intersection points of $Z_{2}$ and $Z_3$ with the boundary of $B_{\alpha}$ must remain confined in their respective arcs delimited by the $r_{[p_i,p_j]}$, e.g.~$Z_2$ has one intersection with the boundary of $B_{\alpha}$ on the arc between $r_{[p_{2a}, p_{3a}]}$ and $r_{[p_{3b}, p_{2a}]}$, similarly for the others; thus the order of intersections along the boundary --- alternating between $Z_2$ and $Z_3$ --- is conserved. This implies that $Z_2$ and $Z_3$ have an intersection point $y_* \in B_{\alpha}$, characterizing a state $(I_3,Q_{2*},Q_{3*}) \in C_{o-e}$ associated to $(G_{2*},G_{3*}) \; \in \; \st(\bar{y}_{12}) \times \st(y_{13})$ and $y_* = G_{3*} \bar{y}_{23} = G_{2*} E_A \bar{y}_{23}$. Then using property (\ref{eq:niceprop}),
\begin{multline*}
\dist(C_s,C_{o-e}) =
\min_{\substack{Q\in SO(n) \\ (I_3,Q_2,Q_3)\in C_{o-e}}} \left( \|Q-I_3\|^2_F + \|Q-Q_2\|^2_F + \|Q-Q_3\|_F^2\right)^{1/2}\\
\leq \min_{\substack{Q\in SO(n) \\ (I_3,Q_2,Q_3)\in C_{o-e}}} \|Q-I_3\|_F + \|Q-Q_2\|_F + \|Q-Q_3\|_F \\
\leq \|I_3-Q_{2*}\|_F + \|I_3-Q_{3*}\|_F \; \leq \; \|I_3-G_{2*}\|_F + \|I_3-G_{3*}\|_F + 6 \eps \, .
\end{multline*}
By the choice of $\alpha$ and $B_{\alpha} \ni y_*$, we have $\Vert y_*-y_{23} \Vert_2 \leq \tfrac{4\eps}{\beta_0 \sin(\theta/2)}$. Moreover, as $G_{3*}$ is a rotation around $y_{13}$, we have $\|I_3-G_{3*}\|_F = \tfrac{\sqrt{2}}{r_{3*}} \Vert y_*-\bar{y}_{23} \Vert_2$ with $r_{3*} = \Vert y_{13} \times \bar{y}_{23} \Vert_2\,$;
similarly $\|I_3-G_{2*}\|_F = \tfrac{\sqrt{2}}{r_{2*}} \Vert y_*-E_A \bar{y}_{23} \Vert_2$ with $r_{2*} = \Vert \bar{y}_{12} \times (E_A\bar{y}_{23}) \Vert_2\,$. There exist $\beta_1,\gamma_1 > 0$ such that $\eps < \gamma_1$ implies $\min(\Vert \bar{y}_{12} \times (E_A\bar{y}_{23})\Vert_2,\; \Vert y_{13} \times \bar{y}_{23} \Vert_2) > \sqrt{2} \beta_1$. Then
\begin{eqnarray*}
\dist( C_s , C_{o-e} ) \leq \frac{\Vert y_* - \bar{y}_{23} \Vert_2 + \Vert y_* - E_A\bar{y}_{23} \Vert_2}{\beta_1} + 6 \eps
\leq  \frac{2 \Vert y_* - y_{23} \Vert_2 + 6\eps}{\beta_1} + 6\eps \; \leq \; \alpha_2 \eps  
\end{eqnarray*}
for small enough $\eps$, with $\alpha_2 = 6+\tfrac{1}{\beta_1} \, (6 + \tfrac{8}{\beta_0 \, \sin(\theta/2)})$ finite.
\hfill $\Box$


\bibliographystyle{plain}
\bibliography{OutConsArxiv}

\end{document}